\documentclass[reqno,12pt]{amsart}
\usepackage[utf8]{inputenc}
\usepackage{amsmath,amssymb,amsthm,mathrsfs,color,times,textcomp,verbatim,yfonts}
\usepackage{tikz}
\usepackage{tikz-3dplot}
\usepackage{graphicx}
\usepackage{subcaption}
\usepackage[export]{adjustbox}
\usepackage{esint}
\usepackage{xcolor}
\usepackage{array}
\usepackage[colorlinks=true]{hyperref}
\hypersetup{urlcolor=blue, citecolor=red, linkcolor=blue}
\usepackage[square,numbers]{natbib}
\usepackage{geometry}
\theoremstyle{plain}
\newtheorem{theorem}{Theorem}
\newtheorem{proposition}[theorem]{Proposition}

\usepackage{graphicx}
\theoremstyle{definition}
\newtheorem{definition}[theorem]{Definition}

\newtheorem{remark}[theorem]{Remark}

\title[nodal solutions to the Yamabe equation]{Planar doubling nodal solutions to the Yamabe equation with maximal rank}
\author{Yuanli Li}
\address{State Key Laboratory of Mathematical Sciences, Academy of Mathematics and Systems Science, Chinese Academy of Sciences, Beijing 100190, China.}
\email{liyuanli@amss.ac.cn}
\author{Liming Sun}
\address{State Key Laboratory of Mathematical Sciences, Academy of Mathematics and Systems Science, Chinese Academy of Sciences, Beijing 100190, China.}
\email{lmsun@amss.ac.cn}
\date{\today}

\begin{document}

\begin{abstract}
This article constructs two families of nodal solutions to the Yamabe equation, each concentrating along two planar circles. One family is conformally equivalent to the one previously obtained by Medina–Musso. The second family is a twisted variant of the first; it is new and is derived from ansatzes that are not Kelvin invariant, in contrast to a standard assumption in earlier works. In addition, in dimension 3 these solutions attain maximal rank. By means of a continuous family of conformal transformations, we then analyze the interaction of the two circles, which display a crossing phenomenon reminiscent, in some sense, of leap-frogging behavior in vortex dynamics.
\end{abstract}

\maketitle
%%%%%%%%%%%%%%%%%%%%%%%%%%%%%%%%%%%%%%%%%%%%%%%%%%%%%%%%%%
\section{Introduction}
Consider the problem
\begin{equation}\label{main-Yamabe}
    -\Delta u=\gamma|u|^{p-1}u\quad \text{in}\quad \mathbb{R}^n,\quad \gamma:=\frac{n(n-2)}{4},\ u\in\mathcal D^{1,2}(\mathbb{R}^n);
\end{equation}
where $n\geq3$, $p=\frac{n+2}{n-2}$, and $\mathcal D^{1,2}(\mathbb{R}^n)=\{u\in L^{p+1} (\mathbb{R}^n):\nabla u\in L^2(\mathbb{R}^n)\}$. It is easy to see that any solution $u$ is a critical point of the energy functional
\begin{align}\label{def:E}
    \mathcal{E}(u)=\frac12\int|\nabla u|^2-\frac\gamma{p+1}\int |u|^{p+1}.
\end{align}

When $u>0$, problem \eqref{main-Yamabe} arises in the classical Yamabe problem or extremal equation for Sobolev inequality. For positive or sign-changing $u$ problem \eqref{main-Yamabe} corresponds to the steady state of the energy-critical focusing nonlinear wave equation
\[\partial_{tt}u-\Delta u-\gamma|u|^{p-1}u=0,\quad  \text{on }\mathbb R\times\mathbb R^{n},\]
These are classical problems that have attracted the attention of many researchers, for instance \cite{Duy12,Duy11,Duy16,Kenig08}. 

The positive solutions to \eqref{main-Yamabe} are well understood, thanks to classical works by Aubin, Talenti and Caffarelli–Gidas–Spruck (see \cite{AUBIN1976,CGS1989,TALENTI1976}), in which they prove that all positive solutions can be written as 
\begin{align}\label{def:U}
    U_{\mu,\xi}(y)=\left(\frac{2\mu}{\mu^2+|y-\xi|^2}\right)^{\frac{n-2}2}.
\end{align}
These functions are called \textit{Aubin-Talenti bubbles}, which serves as building blocks in the following construction. \par

In contrast, much less is known about sign-changing solutions. In 1986, Ding established the existence of such solutions invariant under $O(2)\times O(n-2)$ using the Ljusternik-Schnirelman category method in \cite{DING1986}, but the work did not present any further information on the solution.

Since 2010, several semi-explicit nodal solutions have been constructed. del Pino-Musso-Pacard-Pistoia \cite{DELPINO10} build many nodal solutions on $\mathbb{S}^n$ which exhibit concentration patterns of their energy densities along special manifolds of $S^n$. By stereographic projection,  \cite{DELPINO10} yields many solutions on $\mathbb{R}^n$. The construction does not include the 3-dimensional case. In a subsequent paper \cite{DELPINO2011}, they focus on a particular case which concentrates on a great circle. Their ansatzes consist of the standard bubble $U$ (cf.\,\eqref{def:U}) surrounded by sufficiently large $k$ negative scaled copies of $U$ arranged along the vertices of a $k$-regular polygon in $x_1x_2$-plane.
As $k\to \infty$, the nodal solution concentrates at the unit circle in $x_1x_2$-plane.  It is called the \textit{de-singularization of equator} or named as \textit{crown solution} in \cite{DELPINO2011}. It adopts the inner-outer gluing method and is able to cover the 3-dimensional case. Musso and Wei \cite{MUSSO15} show that the crown solution is non-degenerate in the sense of Duychaerts-Kenig-Merle (see \cite{Duy16}). Such three dimensional crown solution plays an important role in the Brezis-Nirenberg problem (see \cite{sun2025brezis,sun2025small}).

The crown solution does not have maximal rank, see the definition below for being maximal-rank. Two subsequent works find the maximal-rank nodal solutions.  The first one is Medina-Musso-Wei \cite{MEDINA2019}, which uses the ansatzes with the configuration of grids on $S^1\times S^1$ , it is the \textit{de-singularization of two equators}. Such solution has maximal rank in $\mathbb{R}^4$. By placing negative bubbles on $S^1$ in each $x_{2i-1}x_{2i}$-coordinate plane $1\leq i\leq n/2$, \cite{MEDINA2019} is able to construct maximal-rank nodal solutions for any even dimensions.

The second one is by \citet{MEDINA2021}, which uses the configuration of vertices of prisms. More precisely, the ansatz is
\begin{align}\label{medina-1}
    v_{*}(y)=U(y)-\sum^{k}_{j=1}\mu^{-\frac{n-2}2{}}U\bigg(\frac{y-\bar\xi_j}{\mu}\bigg)-\sum^{k}_{j=1}\mu^{-\frac{n-2}{2}}U\bigg(\frac{y-{\underline{\xi}_j}}{\mu}\bigg),
\end{align}
where 
\[\bar\xi_j=(re^{i\frac{2\pi j}{k}},\tau,0,...,0);\quad\underline{\xi}_j=(re^{i\frac{2\pi j}{k}},-\tau,0,...,0),\quad 1\leq j\leq k\]
and $r^2+\tau^2+\mu^2=1$. Note that $\bar\xi_j$ ($\underline{\xi}_j$) are on a circle above (below) the $x_1x_2$-plane, where $\tau>0$ denotes the distance of them to the $x_1x_2$-plane. As $k\to \infty$, they show $r\to 1$ and $\tau\to 0$ and  the ansatz concentrates on the unit circle in $x_1x_2$-plane with multiplicity two, thus it is  call a \textit{doubling of the equator} in the $x_1x_2$-plane. Such configuration is of maximal rank in $\mathbb{R}^3$.  Combining the de-singularization of equator and doubling of equator, the work \cite{MEDINA2021} is able to obtain maximal-rank nodal solutions for any odd dimensions.

Now let us be more precise about maximal rank and non-degeneracy. Recall that equation \eqref{main-Yamabe} is invariant under the M\"obius transform. To be more specific, the M\"obius group $\text{Mob}(n)$ is generated by the following transforms:
\begin{itemize}
    \item translation, $T_a(x)=x+a,\text{ where } a\in \mathbb{R}^n$;
    \item scaling, $\Lambda_\lambda(x)=\lambda x,\text{ where } \lambda\in\mathbb R^+$;
    \item inversion (with respect to the unit ball), $J(x):x\to x/|x|^2$ ;
    \item orthogonal transform, $R_O(x)=Ox,\text{ where } O\in \mathcal{O}(n)$ the orthogonal group.
\end{itemize}
Also recall that the M\"obius group $\text{Mob}(n)$, has the dimension of $\frac12(n+1)(n+2)$. 

If $\Phi\in \text{Mob}(n)$ and $u\in \mathcal{D}^{1,2}(\mathbb{R}^n)$, denote $\Phi_*(u)=|\det \Phi'(x)|^{\frac{n-2}{2n}}u(\Phi(x))$. In particular, if $\Phi=J$, $J_*(u)$ has another name Kelvin transformation and denoted as $K(u)=|x|^{2-n}u(x/|x|^2)$. If $Q$ solves \eqref{main-Yamabe}, so does $\Phi_*(Q)$. Differentiating these family of solutions near $Q$ generates the vector space
\begin{equation*}
T_Q(\text{Orb}_Q(\text{Mob}(n)))=
\text{span}\left\{
\begin{aligned}
    &(2-n)x_jQ+|x|^2\partial_jQ-2x_jx\cdot \nabla Q,\ \partial_jQ,\ 1\leq j\leq n\\
    & x_i\partial_jQ-x_j\partial_iQ,\ 1\leq i<j\leq n,\ \frac{n-2}2Q+x\cdot\nabla Q
\end{aligned}
\right\}.
\end{equation*}
It is evident that $T_Q(\text{Orb}_Q(\text{Mob}(n)))\subset \mathcal{Z}_Q$
where $\mathcal{Z}_Q$ is the kernel of the linearized operator $L_Q:=-\Delta-\gamma p|Q|^{p-2}Q$ around $Q$,
\begin{align*}
\mathcal{Z}_Q:=\{f\in \mathcal{D}^{1,2}(\mathbb{R}^n): L_Q f=0\}.
\end{align*}
\begin{definition}
    For a solution $Q$, if 
    \begin{align*}
        \textbf{Rank}(Q):=\dim T_Q(\text{Orb}_Q(\text{Mob}(n)))=\frac{(n+1)(n+2)}2,
    \end{align*}
then it is said to be \textit{of maximal rank}. If $\mathcal{Z}_Q=T_Q(\text{Orb}_Q(\text{Mob}(n)))$, then $Q$ is said to be \textit{non-degenerate}.
\end{definition}

For instance, consider the case in which $Q=U$, the standard bubble, it is non-degenerate, but not of maximal rank. More precisely, $\mathcal{Z}_U=\text{span}\{Z_1,\cdots, Z_n,Z_{n+1}\}$ where 
\[Z_{n+1}=\frac{n-2}2U+x\cdot\nabla U,\quad Z_j=\partial_jQ,\ j=1,...,n,\]
they form an orthogonal basis for the solution space of the operator $\Delta+pU^{p-1}$. Thus $\textbf{Rank}(U)=n+1$.

%As one may have noticed before, this is exactly a set of basis for $T_U(\text{Orb}_U(\text{Mob}(n)))$. 

%In \cite{MEDINA2021, MEDINA2019}, authors proved that the two families of sign-changing solutions are of maximal rank in even and odd dimensions, respectively, while the solution constructed by \cite{DELPINO2011} is not of maximal rank, as was shown in \cite{MUSSO15}.
%Two subsequent works find the maximal-rank nodal solutions. The first one is \cite{MEDINA2019}, which uses the ansatzes with the configuration of grids on $S^1\cup S^1$ , it is the desingularization of the Clifford torus. The second one is \citet{MEDINA2021}, which uses the configuration of vertices of prisms, it is called the doubling of the equitorial, see \citet{DELPINO10, DELPINO2011, MEDINA2021, MEDINA2019}. In particular, \citet{DELPINO2011} constructed a solution with $D_k\times O({n-2})$ invariance, while \citet{MEDINA2019} constructed a solution with the invariance under $D_k\times D_l\times O(n-4)$, where $D_j$ is the dihedral group. These groups present as the subgroup of the conformal group, i.e., the M\"obius group.\par

The motivation of this paper is twofold. First, we aim to construct new nodal solutions of maximal rank, with particular emphasis on identifying configurations that apply uniformly to both odd and even dimensions. Such a unified framework has been notably absent from previous constructions and is crucial for a deeper structural understanding of these solutions. Second, all existing ansatzes are Kelvin invariant, a symmetry that plays a central role in the inner–outer gluing method. While this invariance has been essential in all known constructions, it remains an open question whether it can be relaxed, and exploring this limitation is key to extending the scope of the method.

%To construct a solution of maximal rank, when the dimension is even, one may consider the configuration given by \cite{MEDINA2019}, when the dimension is odd, one can use that in \cite{MEDINA2021}. 

% In this paper, we give a configuration that works for all dimensions $n\geq 3$. In fact, as noted earlier, these constructions rely on the assumption of Kelvin invariance. In this article, we present a new planar configuration that does not possess Kelvin invariance yet still achieves maximal rank.

Our main theorem gives an affirmative answer to both questions. To state the main theorem, we first define the ansatzes
\begin{align}\label{def:u*m}
    u_{*,m}=U(y)-\sum^{k}_{j=1}\mu_m^{-\frac{n-2}2{}}U(\frac{y-\bar\xi_{j,m}}{\mu_m})-\sum^{k}_{j=1}\lambda_m^{-\frac{n-2}{2}}U(\frac{y-{\hat\xi_{j,m}}}{\lambda_m})
\end{align}
where 
\begin{align}\label{main-xi}
    \bar\xi_{j,m}=(r_me^{i\frac{2\pi j}{k}},0,...,0);\quad\hat\xi_{j,m}=(R_me^{i\frac{\pi (2j+m)}{k}},0,...,0),\ m=0,1;
\end{align}
\begin{theorem}\label{thm:main}
For $k$ sufficiently large, there exists a finite-energy solution $u$ to the Yamabe equation, which is of the form 
\[u_m(y)=u_{*,m}(1+o_k(1))\]
with $o_k(1)\to 0$ uniformly as $k\to \infty$. 
The parameters $\mu_m,r_m,\lambda_m, R_m$ are determined through the following relations: there exists a fixed positive number $\eta<1$ and some $\eta<l_m,t_m<\eta^{-1}$ such that the $\mu_m$ and $d_m:=R_m-r_m$ 
\[\mu_m=\frac{l_m^{\frac2{n-2}}}{k^2},\ d_m=\frac{t_m}{k^{\frac{n-3}{n-2}}},\ n\geq4;\quad \mu_m=\frac{l^2_m}{k^2\log^2k},\ d_m=\frac{t_m}{{\log k}},\ n=3;\]
and 
% \begin{align*}
%     s
% \end{align*}
\begin{align}\label{main-K}
    \lambda_m =\frac{\mu_m }{\mu_m^2+r_m^2};\quad R_m=\frac{r_m}{\mu_m^2+r_m^2}.
\end{align}
% Moreover, there exists a fixed positive number $\eta<1$ and some $\eta<l_m,t_m<\eta^{-1}$ such that the $\mu_m$ and $d_m:=R_m-r_m$ can be determined by 
%Here and $d_m:=R_m-r_m$ satisfy the relation
%\[\mu_m=\frac{l_m^{\frac2{n-2}}}{k^2},\ d_m=\frac{t_m}{k^{\frac{n-3}{n-2}}},\ n\geq4;\quad \mu_m=\frac{l^2_m}{k^2\log^2k},\ d_m=\frac{t_m}{{\log k}},\ n=3;\]
%Here $\eta$ is a fixed positive number, and the residue uniformly converges to 0 on any compact subset of $\mathbb{R}^n$.\par
Moreover, we have the $\text{\textbf{Rank}}(u_m)=4n-2$. In particular, the solutions are of maximal rank when $n=3$.
\end{theorem}
% {\color{blue} change $\hat \xi_j$ to $\hat \xi_j(s)$ 

% We need to picture to illustrate $m=0$ and $m=1/2$.
% }
%In the following context, we denote $d_m:=R_m-r_m$, and $\mu_m=k^{-2}\hat l_m^{\frac2{n-2}}$.

\begin{remark}
 From the relationship of $\lambda_m, \mu_m,R_m,r_m$, we obtain an equation for $r_m$:
\[r^3_m+d_mr^2_m+(\mu_m^2-1)r_m+d_m\mu^2_m=0.\]
Then we see that $r_m=1-\frac12(\mu_m^2+d_m)+O(\mu_m^3)+O(d_m^2)$. Thus $r_m^2+\mu_m^2<1$ and consequently $\lambda_m^2+R_m^2=1/(r_m^2+\mu_2)>1$. 
% If we denote $\lambda_m=k^{-2}\hat l_m^{\frac2{n-2}}$ for $n\geq 4$ and $\lambda_m=\hat l_m^2/(k^2\log^2k)$ if $n=3$, then
% \begin{equation}
%     \hat l_m=l_m(r_m^2+\mu_m^2)^{\frac{2-n}2}=l_mr_m^{2-n}(1+(\mu_m/r_m)^2)^{\frac{2-n}2}=l_m(1+o_k(1)).
% \end{equation}
\end{remark}
Let us make some discussion about the theorem. Note that we construct two families of nodal solutions to the Yamabe equation, $m=0$ and $m=1$. For both of them, they have two sets of highly concentrated negative bubbles, one locates at an inner circle of radius $r_m$ and the other one locates at outer circle  $R_m$. These two circles are on the same $x_1x_2$-plane. As $k\to \infty$, One can see that the ansatzes concentrate on the unit circle in $x_1x_2$-plane with multiplicity two. It is the \textit{planar doubling of the equator}. We emphasize that it is different from the configuration in \eqref{medina-1} whose two circles are not on the same plane. 
%In higher dimensions, one can combine this planar doubling in each $x_{2i-1}x_{2i}$-coordinate plan $1\leq i\leq \lfloor n/2\rfloor$ to get maximal rank nodal solutions, which unify the odd and even dimensional cases.

% The ansatzes of the solutions
% $$u_{m,*}(y)=U(y)-\sum^{k}_{j=1}\mu_m^{-\frac{n-2}2{}}U(\frac{y-\bar\xi_{j,m}}{\mu_m})-\sum^{k}_{j=1}\lambda_m^{-\frac{n-2}{2}}U(\frac{y-{\hat\xi_{j,m}}}{\lambda_m}),$$ We call these approximations the ansatzes of the solutions, which are shown in Figure \ref{fig:1}. 

Two families of nodal solutions have distinctive features. We point out that the ansatzes with $m=0$ are Kelvin invariant, the ansatzes with $m=1$ \textit{are not} because the extra rotation of angle $\pi/k$ (see Figure \ref{fig:1}). In fact, condition \eqref{main-K} exactly implies that two circles will be swapped to each other under the transform $\mathcal K$. Consequently, the bubbles on inner circle is mapped to ones on outer circle under Kelvin transformation when $m=0$, while not for $m=1$. Nevertheless, the $m=1$ cases retain certain symmetry. We define the rotation $$\mathcal{T}=\text{diag}\{e^{-i\frac{\pi}{k}},1,...,1\},\ \mathcal{T}f(x):=f(\mathcal{T}x),$$ also denote $K$ the standard Kelvin transform, then the solution \textit{is} invariant under the linear transform $\mathcal K:=\mathcal{T}\circ K=K\circ\mathcal{T}.$
%To ensure that such symmetry is satisfied, the relation of $\mu,\lambda, d, r$ is imposed on the approximate solution. 

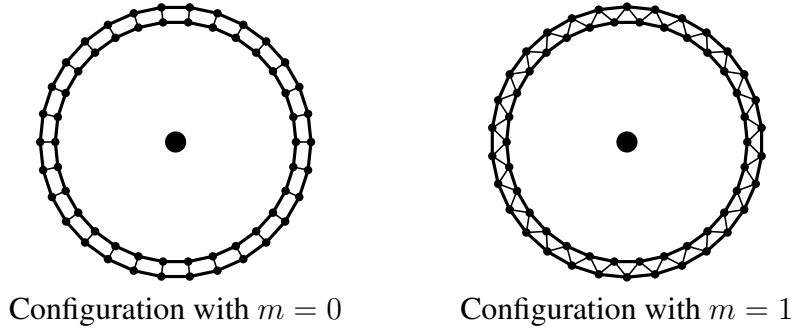
\begin{figure}
\begin{center}
\begin{tikzpicture}[scale=2.0, line cap=round, line join=round, thick]
  
  % ====== 左侧：顶点共线的同心十边形 ======
  \begin{scope}[xshift=-1.5cm]
    % 参数设置
    \def\n{30}                % 边数
    \def\angle{12}            % 十边形的中心角
    \def\rA{0.8}              % 内十边形半径
    \def\rB{0.9}              % 外十边形半径
    
    % 绘制中心点
    \coordinate (O) at (0,0);
    \fill[black] (O) circle (2pt);
    
    % 生成内十边形顶点
    \foreach \i in {1,...,\n} {
      \pgfmathsetmacro{\theta}{(\i-1)*\angle}
      \coordinate (A\i) at (\theta:\rA);
    }
    
    % 生成外十边形顶点
    \foreach \i in {1,...,\n} {
      \pgfmathsetmacro{\theta}{(\i-1)*\angle}
      \coordinate (B\i) at (\theta:\rB);
    }
    
    % 绘制内十边形（黑色实线）
    \draw[black, line width=1.2pt] (A1) 
      \foreach \i in {2,...,\n} {-- (A\i)}
      -- cycle;
    
    % 绘制外十边形（黑色实线）
    \draw[black, line width=1.2pt] (B1)
      \foreach \i in {2,...,\n} {-- (B\i)}
      -- cycle;
    
    % 绘制连接内、外十边形对应顶点的线（黑色细线）
    \foreach \i in {1,...,\n} {
      \draw[black, line width=0.6pt] (A\i) -- (B\i);
    }
    
    % 标记顶点
    \foreach \i in {1,...,\n} {
      \fill[black] (A\i) circle (0.8pt);
      \fill[black] (B\i) circle (0.8pt);
    }

    \node[above] at (0, \rB -2.2) {Configuration with $m=0$};
  \end{scope}
  
  % ====== 右侧：旋转的同心十边形 ======
  \begin{scope}[xshift=1.5cm]
    % 参数设置
    \def\n{30}                % 边数
    \def\angle{12}            % 十边形的中心角
    \def\rotateAngle{6}      % 旋转角度 = 2π/20弧度 = 18度
    \def\rA{0.8}              % 内十边形半径
    \def\rB{0.9}              % 外十边形半径
    
    % 绘制中心点
    \coordinate (O) at (0,0);
    \fill[black] (O) circle (2pt);
    
    % 生成内十边形顶点（未旋转）
    \foreach \i in {1,...,\n} {
      \pgfmathsetmacro{\theta}{(\i-1)*\angle}
      \coordinate (A\i) at (\theta:\rA);
    }
    
    % 生成外十边形顶点（旋转18度）
    \foreach \i in {1,...,\n} {
      \pgfmathsetmacro{\theta}{(\i-1)*\angle + \rotateAngle}
      \coordinate (B\i) at (\theta:\rB);
    }
    
    % 绘制内十边形（黑色实线）
    \draw[black, line width=1.2pt] (A1) 
      \foreach \i in {2,...,\n} {-- (A\i)}
      -- cycle;
    
    % 绘制外十边形（黑色实线）
    \draw[black, line width=1.2pt] (B1)
      \foreach \i in {2,...,\n} {-- (B\i)}
      -- cycle;
    
    % 绘制连接内、外十边形顶点（交替连接，形成五边形星形图案）
    \foreach \i in {1,...,\n} {
      \pgfmathsetmacro{\next}{mod(\i, \n) + 1}  % 下一个顶点的索引
      \draw[black, line width=0.6pt] (A\i) -- (B\i);
      \draw[black, line width=0.6pt] (A\next) -- (B\i);
    }
    
    % 标记顶点
    \foreach \i in {1,...,\n} {
      \fill[black] (A\i) circle (0.8pt);
      \fill[black] (B\i) circle (0.8pt);
    }

    \node[above] at (0, \rB- 2.2) {Configuration with $m=1$};
  \end{scope}
  
  % 添加标题
   % \node[above, align=center] at (1.5,0,4) {\textbf{Fig. 1  }};
  
\end{tikzpicture}
\end{center}
    \caption{The planar twin solutions}
    \label{fig:1}
\end{figure}

The construction can be extended to higher even dimensions to produce some more solutions of maximal rank. one can analogously place bubbles in the $x_{2i-1}x_{2i}$-plane, where $i\leq\lfloor n/2\rfloor$. One can get maximal-rank solution for even and odd dimensions at the same time by using this approach. 
%We will address this topic in the last section.

% \begin{remark}
% Nota bene, the solutions we constructed above are of maximal rank, but whether they are maximal is unknown. To be specific, we need to introduce the notion of non-degeneracy. A solution $Q$ to the Yamabe equation is said to be \textit{non-degenerate} if the kernel of the linearized operator at $Q$, namely $L=\Delta+\gamma p|Q|^{p-2}Q$, is spanned entirely by vectors from the space $S(Q)$.
% A solution is called maximal if it is non-degenerate and of maximal rank. It should be noted that the family of solutions constructed by the authors is indeed of maximal rank, but it is not certain whether they are maximal. In fact, whether the solutions constructed in the main theorem, as well as the maximal rank solutions constructed in section 7, are non-degenerate remains an open question. However, the authors conjecture that these solutions are maximal.

% \end{remark}

We also explore the conformal transformation of the solution constructed in Theorem \ref{thm:main}. By making certain translation and Kelvin transformation at point $a\,\mathbf{e}_3$ which belongs to
$x_3$-axis, we can get two families of new ansatzes whose two circles are above and below $x_1x_2$-plane respectively. These two circles will enlarge/shrink and ascend/descend along the movement of $a$ (cf. Figure \ref{fig:2}). As $a$ run from $-\infty$ to $+\infty$, two circles will exhibit a crossing phenomenon. This resembles the leap-frogging in vortex dynamics in some sense. 

% When $a=\pm 1$, the two circles have the same radius and are at the same distance to the $x_1x_2$-plane. One can see that the conformal counterpart of the family corresponds to $m=0$ has exact configuration to the one in \cite{MEDINA2021}. This suggests that the family of $m=0$ might be conformally equivalent to the one in \cite{MEDINA2021}. However, the construction in \cite{MEDINA2021} only provides the uniqueness of the nodal solution in its neighborhood, we are not sure the conformal counterpart of the family of $m=0$ lies in the neighborhood. At least, the family corresponds to $m=1$ can not be conformally equivalent to the one in \cite{MEDINA2021}. Whether the solutions in Theorem \ref{thm:main} are non-degenerate remain an open problem.

When $a=\pm 1$, the two circles have the same radius and are at the same distance from the $x_1x_2$-plane. Surprisingly, one observes that the conformal counterpart of the family corresponding to $m=0$ has exactly the same configuration as \eqref{medina-1} in \cite{MEDINA2021}. Since both the construction here and there are using contraction mapping theorem in a small neighborhood of $u_{*,0}$ and its conformal counterpart, when $k$ is large enough, two families of nodal solution are conformally equivalent.
%However, the construction in \cite{MEDINA2021} only ensures the uniqueness of the nodal solution in a neighborhood, and it is unclear whether the conformal change of the $m=0$ family lies inside this neighborhood. 
On the other hand, the family with $m=1$ cannot be conformally equivalent to the one in \cite{MEDINA2021} (cf.\,Theorem \ref{thm:2}). 

In fact, when considered in the sense of conformal equivalence, the motivation for constructing such an ansatz is even clearer. In the case where $a=\pm1$, the Kelvin invariance of the ansatz $m=0$ exactly corresponds to the evenness with respect to $x_3$ for the family of Musso-Medina. Meanwhile, the symmetry of the ansatz $m=1$ corresponds to a reflection across $x_1x_2$-plane and combined a rotation in that plane. Indeed, there is a similar  geometric transformation called rotoreflection (also known as improper rotation) in crystallography. 

% in the point group theory of crystals, there is a symmetry called rotoreflection, the counterpart of mirror reflection is considered to be retoref of crystals, lection, which corresponds to the novel symmetry mentioned above under the inverse map.
% The method we adopt here in the proof the main theorem is Lyapunov-Schmidt reduction, which is a classical finite-dimensional reduction. Following the process performed by \cite{DELPINO2011,MEDINA2021,MEDINA2019}, the process is clear. First, constructing a family of approximate solutions, known as ansatzes; then solving the small error terms by using linear theory developed by \cite{DELPINO2011} and fixed point theorem, and the problem is reduced to a equation; finally, solving the finite equation system by asymptotic expansion.

The method we adopt here in the proof the main theorem is inner-outer gluing, which has been established by previous work \cite{DELPINO2011,MEDINA2021,MEDINA2019}. This method entails separating the problem into two components: an inner problem and an outer problem. Solving each requires imposing discrete symmetries on the ansatzes, which have conventionally been chosen as even symmetry, Kelvin invariance, and discrete rotation invariance in the previous work. Our work shows that the Kelvin invariance can be relaxed to some extent. Determining whether the solutions in Theorem \ref{thm:main} are non-degenerate remains an open question.

In the following sections, the argument is organized as follows. In Section 2, we will give a formal calculation to demonstrate construction of ansatzes, and in Section 3, we analyzed the conformal properties, and compared that to the previous works. The rigorous proof starts from Section 4, where a first-step approximation is established. In Section 5, we establish the symmetry decomposition of the error term, which is needed for the following section. In Section 6, we find the parameter to finish the proof of main theorem.  
\section{Conformally equivalent counterparts of the solutions}

%As we have pointed out that, the two solutions $u_0,u_1$ constructed in Theorem \ref{thm:main} is based on two planar circles, while the solution $v_0$ constructed by \cite{MEDINA2021} and the solution $v_1$ as pointed out in \eqref{medina} are based on two circles not in the same plane. Although the configuration seems to be different, it might be possible that a conformal change of nodal solution in Theorem \ref{thm:main} can be the one in \cite{MEDINA2021}.  We must exclude this case. Therefore, 
Before diving into the proof of Theorem \ref{thm:main}, we shall exhibit some conformal changes of the solutions and explore the connection of them with the solution constructed by \cite{MEDINA2021} in this section. 
\begin{theorem}\label{thm:2}
     The family of solutions  $u_1$ constructed in Theorem \ref{thm:main} is not conformally equivalent to $u_0$ there or the one in \cite{MEDINA2021} when $k$ is large enough. 
     %That is, there is no M\"obius transform $\Phi$ such that $\Phi_*(u_1)=u_0$ or . %Neither $u_0$ and $u_1$ are conformally equivalent to $v_0$ constructed in \cite{MEDINA2021}.
\end{theorem}
\begin{proof}
    Recall that any M\"obius transform $\Phi$ on $\mathbb{R}^n$ is either a similarity transformation $\Phi(x)=\lambda O x+a$ or a transformation in the form $\Phi=\lambda O(x-a)/|x-a|^2+b$ for some $\lambda>0$, $O\in \mathcal{O}(n)$ and $a,b\in \mathbb{R}^n$, for instance see \cite[p. 75]{reshetnyak2013stability}. In either case, the conformal transformation will rotate the two circles with the same angle. Therefore $u_1$ can not be conformally equivalent to $u_0$ the one in \cite{MEDINA2021} when $k$ is large enough.
    %By suitably modifying $a,b$, we can assume that $O$ is the identity matrix.

    % Using the above fact, it is easy to show that $u_0$ is not conformally equivalent $u_1$, $v_0$ is not conformally equivalent to $v_1$.
\end{proof}
    
%     In the first case, since any similar transformation keeps the two planar circles of $u_m$ in a same plane, thus it can not be mapped to $v_m$. In the latter case, pay attention to the even symmetry on $x_2,\cdots, x_n$, the dihedral symmetry in $x_1x_2$-plane and radial symmetry in $x_3,\cdots, x_n$ of $u_{m,*}$ and $v_{m,*}$ and the configuration of $v_{m,*}$, then $a,b$ must be on the $x_3$-axis and $O$ belong to the dihedral group. Without loss of generality, we assume that $O$ is the identity matrix. We abuse the notation by denoting $a,b\in \mathbb{R}$. Thus the only possible M\"obius transform  
%     \begin{align*}
%         \Phi_*(u_m)=\lambda^{\frac{n-2}{2}}|x-a\mathbf{e}_3|^{2-n}u_m\Big(\frac{\lambda(x-a\mathbf{e}_3)}{|x-a\textbf{e}_3|^2}+b\mathbf{e}_3\Big)
%     \end{align*}
%     Note that $u_{m}$ is approximated by $u_{m,*}$ and $v_m$ is approximated by $v_{m,*}$, then for the standard bubble $U$ in the origin, one has $\Phi_*(U)$ is very near to $U$. Since $u_m$ and $v_m$ all take a non-degenerate local maximum at $0$, we must have $\Phi_*(U)=U$. This implies that $b=a$ and $\lambda=1+|a|^2$. Thus 
Next, we will explore the conformally equivalent counterparts of $u_m$ for $m=0,1$. In order to keep two circles being concentric, we consider a family of conformal transform parametric by the points on the $x_3$-axis, namely
$$ \Phi_a=T_{a}\circ\Lambda_{(1+a^2)} \circ J\circ T_{-a}$$
where $\Lambda, T,J$ are the scaling, translation and inversion defined in the introduction respectively.
% The idea of defining $\mathbb{I}_a$ comes from pulling the solution to $\mathbb{S}^n$ by stereographic projection, and then change the north pole and map the solution back. The scaling $\Lambda_{(1+a^2)}$ and the translation $T_{-K(a)}$ is used to keep the standard bubble at the origin after the transformation.
As a simple calculation demonstrates, for an Aubin-Talenti bubble $U_{\mu,\xi}$, we have 
\[J_*(U_{\mu,\xi})=K(U_{\mu,\xi})=U_{\bar \mu,\bar \xi}\quad \text{where}\quad \bar \mu=\frac{\mu}{\mu^2+|\xi|^2},\quad \bar \xi=\frac{\xi}{\mu^2+|\xi|^2}.\]

By applying the relationship, now we may consider more carefully about the geometric picture of the configurations. After the conformal transformation $\Phi_a$, we get a new configuration, say $u_{*,m,a}(y)=(\Phi_a)_*(u_{*,m})$,
\[u_{*,m,a}(y)=U(y)-\sum_{j=1}^n\mu_{m,a}^{2-n}U\bigg(\frac{y-\bar\xi_{j,m,a}}{\mu_{m,a}}\bigg)-\sum_{j=1}^n\lambda_{m,a}^{2-n}U\bigg(\frac{y-\hat\xi_{j,m,a}}{\lambda_{m,a}}\bigg),\]
where
\[\bar \xi_{j,m,a}=\frac{(\bar\xi_{j,m}-a\mathbf{e}_3)(1+a^2)}{\mu_m^2+r_m^2+a^2}+a\mathbf{e}_3,\ \hat \xi_{j,m,a}=\frac{(\hat\xi_{j,m}-a\mathbf{e}_3)(1+a^2)}{\lambda_m^2+R_m^2+a^2}+a\mathbf{e}_3,\]
and $\mathbf{e}_3=(0,0,1,0,...,0)^T$, also
\[\mu_{m,a}=\frac{\mu_m(1+a^2)}{\mu_m^2+r_m^2+a^2};\ \lambda_{m,a}=\frac{\lambda_m(1+a^2)}{R^2_m+\lambda^2_m+a^2}.\]

To be more precise, 

$\bullet $ The center bubble $U(y)$ remains the same under $\Phi_a$.

$\bullet $ The inner-layer satellite bubbles $\mu_m^{\frac{2-n}2{}}U\big(\mu_m^{-1}({y-\bar\xi_{j,m}})\big)$ are changed to new bubbles $\mu_{m,a}^{\frac{2-n}2}U\big(\mu_{m,a}^{-1}{(y-\bar\xi_{j,m,a}})\big)$. These bubbles will concentrate on new points
\begin{align*}
    \bar\xi_{j,m,a}=\frac{r_m(1+a^2)}{(\mu_m^2+r_m^2+a^2)}\bigg(\cos\frac{\pi(2j+m)}{k}\mathbf{e}_1+\sin\frac{\pi(2j+m)}{k}\mathbf{e}_2+\frac {a(\mu_m^2+r_m^2-1)}{r_m(1+a^2)}\mathbf{e}_3\bigg).
\end{align*}
which lie on a new circle. For simplicity, we refer such circle as the \textit{bar circle}.

$\bullet $ The outer-layer satellite bubbles $\lambda_m^{\frac{2-n}2}U\big(\lambda_m^{-1}({y-\hat\xi_{j,m}})\big)$ are changed to new bubbles $\lambda_{m,a}^{\frac{2-n}2}U\big(\lambda_{m,a}^{-1}{(y-\hat\xi_{j,m,a}})\big)$. These bubbles will concentrate on new points
\begin{align*}
    \hat\xi_{j,m,a}=\frac{R_m(1+a^2)}{(\lambda_m^2+R_m^2+a^2)}\bigg(\cos\frac{\pi(2j+m)}{k}\mathbf{e}_1+\sin\frac{\pi(2j+m)}{k}\mathbf{e}_2+\frac {a(\lambda_m^2+R_m^2-1)}{R_m(1+a^2)}\mathbf{e}_3\bigg).
\end{align*}
which lie on a new circle. For simplicity, we refer such circle as the \textit{hat circle}.

There are two special cases. The first one is when the bar circle and hat circle has the same radius, the second one is when the bar circle and hat circle has the same distance to the $x_1x_2$-plane. We want to find the exact $a$ when these two special cases happen. 

For the first one, one must have
\begin{align}\label{radius-two}
    \frac{r_m(1+a^2)}{\mu_m^2+r_m^2+a^2}=\frac{R_m(1+a^2)}{\lambda_m^2+R_m^2+a^2}.
\end{align}
Using \eqref{main-K}, we can show that the above holds only when $a=\pm 1$.
% \begin{align*}
%     \frac{\mu}{\mu^2+r^2}=\frac{\lambda}{\lambda^2+R^2},\quad \frac{r}{\mu^2+r^2}=\frac{R}{\mu^2+R^2}
% \end{align*}

% The length of its projection on $x_1x_2$ plane is
% $\bar r_m=r_m{(\mu^2_m+r^2_m+a^2)}^{-1}.$
% Similarly, for the outer layer, we also have 
% $\bar R_m=R_m(\lambda^2_m+R^2_m+a^2_m)^{-1}=R_m((\mu_m^2+r^2_m)^{-1}+a^2)^{-1}.$ So far, we can obtain some interesting cases. When $\bar r_m=\bar R_m$, we have
% \[a=\pm \sqrt{\frac{r_mR_m}{R_m-r_m}\bigg(\frac{\lambda_m^2+R_m^2}{R_m}-\frac{\mu_m^2+r_m^2}{r_m}\bigg)}=\pm1.\]

%We check the location of the central bubble, that is $\textbf{0}$. 
For the second case, one must have
\begin{align}\label{height-two}
    \frac{a(\mu_m^2+r_m^2-1)}{\mu_m^2+r_m^2+a^2}=-\frac{a(\lambda_m^2+R_m^2-1)}{\lambda_m^2+R_m^2+a^2}.
\end{align}
Using \eqref{main-K}, we can show that the above holds only when $a=\pm 1$ or 0.

When \( a = \pm 1 \), both circles have the same radius and are equidistant from the \( x_1x_2 \)-plane. In this case, the conformal counterpart of the family with \( m = 0 \) exactly matches the configuration described in \cite{MEDINA2021}, suggesting that this family may be conformally equivalent to the one in that reference. However, since the construction in \cite{MEDINA2021} only guarantees the uniqueness of the nodal solution within its neighborhood, it remains uncertain whether the conformal counterpart of the \( m = 0 \) family actually lies in that neighborhood.

% Thus the only possible case $(\Phi_a)_*(u_m)=v_m$ is when $a=\pm 1$.
% %Moreover, the reader should be careful to compare the case $a=\pm1$ to that constructed by \cite{MEDINA2021}. 
% However, a simple calculation using the previously given coordinates shows that the distance of the two layers, say $\tau_{m}'$, has the order $\tau_{m,*}\sim d_m\sim k^{\frac{3-n}{n-2}}$. However, according to the main theorem in \cite{MEDINA2021}, as well as the tautological calculation for the twin solution of that one, we see that the distance is $\tau_m\sim k^{\frac{3-n}{n-1}}$, which is not even of the same order with $\tau_m'$. Therefore, $u_m$ can not be conformally equivalent to $v_m$ when $k$ is large enough.

% If we want the average of two layers to be exactly 0, then
% \[\frac{2}{a^2+1}=\frac1{a^2+r_m^2+\mu_m^2}+\frac1{a^2+R_m^2+\lambda^2}=\frac1{a^2+r_m^2+\mu_m^2}+\frac1{a^2+\frac1{r_m^2+\mu_m^2}},\]
% then
% \begin{equation*}
% \begin{aligned}
% &\frac1{a^2+r_m^2+\mu_m^2}+\frac1{a^2+\frac1{r_m^2+\mu_m^2}}-\frac2{a^2+1}\\
% =\ &\frac{(1-a^2)(r_m^2+\mu_m^2-1)^2}{(1+a^2)(1+a^2(r_m^2+\mu_m^2))(a^2+r_m^2+\mu_m^2)}\geq0,
% \end{aligned}
% \end{equation*}
% the equality holds iff $|a|=1$, which shows that the central bubble is still the center of the configuration.

Indeed we can explore the animation of these moving ansatzes along when $a$ goes from $-\infty$ to $+\infty$ (see Figure \ref{fig:2} for case $m=0$).. Recall that we refer to the circle containing $\hat\xi_{i,m,a}$ as the hat circle and the one containing \(\bar\xi_{i,m,a}\) as the bar one. As $a$ increases from $-\infty$ to 0, it is easy to see from \eqref{radius-two} that the radius of bar circle is increasing and hat circle is decreasing, and they have the same radius when $a=-1$. If $a
$ increase from 0 to $+\infty$, then bar circle is shrinking and hat circle is expanding, and they have the same radius when $a=1$. 
%We need to have a careful investigation on the movement on $x_3$ direction of both circles as $a$ increases.

%We now need to investigate the motion of both circles as $a$ increases.
To describe the altitude of these two circles, we consider the functions $\langle \hat\xi_{i,m,a},\mathbf{e}_3\rangle$ and $\langle \bar\xi_{i,m,a},\mathbf{e}_3\rangle$ in \eqref{height-two}. It is easy to see that they reach the critical points when $a=\pm\hat a_*$ and $a=\pm\bar a_*$ respectively, where 
\[\bar a_*=\sqrt{r^2_m+\mu^2_m};\ \hat a_*=\sqrt{R^2_m+\lambda^2_m}.\] A simple comparison shows that $\hat a_*>1>\bar a_*>0$.
%\[\hat a_*>1>\bar a_*>0.\]

Therefore, the two circles approach the $x_1x_2$-plane when $a\to\pm\infty$.
When  $a$  increases from  $-\infty$ , the bar circle ascend from $x_1x_2$-plane and achieve its maximum altitude when $a=\bar a_*$, after that the bar circle will descend and reach the plane when $a=0$. While, the hat circle descend from $x_1x_2$-plane and achieve its lowest altitude when $a=-\hat a_*$, after that the hat circle start to ascend and reach the plane when $a=0$. When $a=-1$, the two circles has the same distance to the plane. When $a$ increase 0 to $+\infty$, the change of altitude of two circles are mirroring to the $a<0$ case.

%expands and the hat circle shrinks until $a=0$. As we have calculated, the bar circle descends and the hat circle rises when $a\leq-\hat a_*$. Then the hat circle, as well as the bar circle, descends when $-\hat a_*<a\leq-\bar a_*$. Specially, when $a = -1$ , they become equal in size, forming the configuration discussed earlier.  After passing through  $a = -\hat a_*$ , the bar circle begins to ascend. Later, when  $a$  reaches 0, we obtain the configuration in our main theorem, at which point the hat circle starts to expand while the bar circle starts to shrink, and the vertical motion directions  of both circles stay the same. At  $a = \bar a_*$ , the bar circle starts to descend again, and at  $a = \hat a_*$ , the hat circle starts to rise again, continuing until  $a$  tends to  $+\infty$ .
Notably, during the increase of $a$, the bar circle passes through the hat circle, creating a crossing phenomenon. If we connect the case $a=\pm\infty$, then the crossing phenomenon will repeat. This resembles the leap-frogging in vortex dynamics in some sense.

 \begin{figure}[ht]
 \centering
 \tdplotsetmaincoords{60}{135} % 设置视角（倾斜视图）

 \begin{tikzpicture}[tdplot_main_coords, line join=round, line cap=round]
  
   % 全局参数调整，使五个图适合一行
   \def\scaleFactor{0.6}
   \def\spacing{3.0}
  
   \begin{scope}[xshift=-1.5*\spacing cm , scale=\scaleFactor]
     % 参数设置
     \def\n{20}            % 边数
     \def\angle{18}        % 十边形的中心角
     \def\rBottom{2.0}     % 下底半径
     \def\rTop{1.5}        % 上底半径（稍小，形成棱台）
     \def\height{0}      % 高度（较小值实现扁平效果）
    
     % 绘制下底面（十边形）
     \foreach \i in {0,...,\n} {
       \pgfmathsetmacro{\theta}{\i*\angle}
       \pgfmathsetmacro{\x}{\rBottom*cos(\theta)}
       \pgfmathsetmacro{\y}{\rBottom*sin(\theta)}
       \coordinate (B\i) at (\x, \y, 0); % B表示底部
     }
    
     % 绘制上底面（十边形）
     \foreach \i in {0,...,\n} {
       \pgfmathsetmacro{\theta}{\i*\angle}
       \pgfmathsetmacro{\x}{\rTop*cos(\theta)}
       \pgfmathsetmacro{\y}{\rTop*sin(\theta)}
       \coordinate (T\i) at (\x, \y, \height); % T表示顶部
     }
    
     % 绘制侧棱
     \foreach \i in {0,...,\n} {
       \draw[black, thick] (B\i) -- (T\i);
     }
    
     % 绘制下底面边
     \foreach \i in {0,...,\n} {
       \pgfmathsetmacro{\next}{mod(\i+1,\n+1)}
       \draw[black, thick] (B\i) -- (B\next);
     }
    
     % 绘制上底面边
     \foreach \i in {0,...,\n} {
       \pgfmathsetmacro{\next}{mod(\i+1,\n+1)}
       \draw[black, thick] (T\i) -- (T\next);
     }
    
     % 标记顶点
      \foreach \i in {0,...,\n} {
       \draw[black, line width=0.6pt, fill=white] (B\i) circle (3pt);
       \draw[black, line width=0.6pt, fill=black] (T\i) circle (3pt);
     }
    
     \node[below, align=center] at (0,0,-1.5) {$a=-\infty$};
   \end{scope}

   % ====== 图1：扁平的十方棱台 ======
 \begin{scope}[xshift=-0.5*\spacing cm, scale=\scaleFactor]
     % 参数设置 - 与图1相反
     \def\n{20}            % 边数
     \def\angle{18}        % 十边形的中心角
     \def\rBottom{1.8}     % 下底半径（小）
     \def\rTop{2.0}        % 上底半径（大）
     \def\height{0.3}      % 高度（较小值实现扁平效果）
    
     % 绘制下底面（十边形）
     \foreach \i in {0,...,\n} {
       \pgfmathsetmacro{\theta}{\i*\angle}
       \pgfmathsetmacro{\x}{\rBottom*cos(\theta)}
       \pgfmathsetmacro{\y}{\rBottom*sin(\theta)}
       \coordinate (B\i) at (\x, \y, 0); % B表示底部
     }
    
     % 绘制上底面（十边形）
     \foreach \i in {0,...,\n} {
       \pgfmathsetmacro{\theta}{\i*\angle}
       \pgfmathsetmacro{\x}{\rTop*cos(\theta)}
       \pgfmathsetmacro{\y}{\rTop*sin(\theta)}
       \coordinate (T\i) at (\x, \y, \height); % T表示顶部
     }
    
     % 绘制侧棱
     \foreach \i in {0,...,\n} {
       \draw[black, thick] (B\i) -- (T\i);
     }
    
     % 绘制下底面边
     \foreach \i in {0,...,\n} {
       \pgfmathsetmacro{\next}{mod(\i+1,\n+1)}
       \draw[black, thick] (B\i) -- (B\next);
     }
    
     % 绘制上底面边
     \foreach \i in {0,...,\n} {
       \pgfmathsetmacro{\next}{mod(\i+1,\n+1)}
       \draw[black, thick] (T\i) -- (T\next);
     }
    
     % 标记顶点
      \foreach \i in {0,...,\n} {
       \draw[black, line width=0.6pt, fill=black] (B\i) circle (3pt);
       \draw[black, line width=0.6pt, fill=white] (T\i) circle (3pt);
     }
    
     \node[below, align=center] at (0,0,-1.5) {$a=-2$};
   \end{scope}

   % ====== 图2：扁平的十方棱柱 ======

   \begin{scope}[xshift=0.5*\spacing cm, scale=\scaleFactor]
     % 参数设置 - 与图2相同
     \def\n{20}            % 边数
     \def\angle{18}        % 十边形的中心角
     \def\radius{2.0}     % 底面半径
     \def\height{0.4}      % 高度（较小值实现扁平效果）
    
     % 绘制下底面（十边形）
     \foreach \i in {0,...,\n} {
       \pgfmathsetmacro{\theta}{\i*\angle}
       \pgfmathsetmacro{\x}{\radius*cos(\theta)}
       \pgfmathsetmacro{\y}{\radius*sin(\theta)}
       \coordinate (B\i) at (\x, \y, 0);
     }
    
     % 绘制上底面（十边形）
     \foreach \i in {0,...,\n} {
       \pgfmathsetmacro{\theta}{\i*\angle}
       \pgfmathsetmacro{\x}{\radius*cos(\theta)}
       \pgfmathsetmacro{\y}{\radius*sin(\theta)}
       \coordinate (T\i) at (\x, \y, \height);
     }
    
     % 绘制侧棱
     \foreach \i in {0,...,\n} {
       \draw[black, thick] (B\i) -- (T\i);
     }
    
     % 绘制下底面边
     \foreach \i in {0,...,\n} {
       \pgfmathsetmacro{\next}{mod(\i+1,\n+1)}
       \draw[black, thick] (B\i) -- (B\next);
     }
    
     % 绘制上底面边
     \foreach \i in {0,...,\n} {
       \pgfmathsetmacro{\next}{mod(\i+1,\n+1)}
       \draw[black, thick] (T\i) -- (T\next);
     }
    
     % 标记顶点
      \foreach \i in {0,...,\n} {
       \draw[black, line width=0.6pt, fill=black] (B\i) circle (3pt);
       \draw[black, line width=0.6pt, fill=white] (T\i) circle (3pt);
     }
    
     \node[below, align=center] at (0,0,-1.5) {$a=-1$};
   \end{scope}

   % ====== 图3：斜视的同心圆 ======
   \begin{scope}[xshift=1.5*\spacing cm, scale=\scaleFactor]
     % 参数设置
     \def\n{20}            % 边数
     \def\angle{18}        % 十边形的中心角
     \def\rBottom{2.0}     % 下底半径
     \def\rTop{1.5}        % 上底半径（稍小，形成棱台）
     \def\height{0}      % 高度（较小值实现扁平效果）
    
     % 绘制下底面（十边形）
     \foreach \i in {0,...,\n} {
       \pgfmathsetmacro{\theta}{\i*\angle}
       \pgfmathsetmacro{\x}{\rBottom*cos(\theta)}
       \pgfmathsetmacro{\y}{\rBottom*sin(\theta)}
       \coordinate (B\i) at (\x, \y, 0); % B表示底部
     }
    
     % 绘制上底面（十边形）
     \foreach \i in {0,...,\n} {
       \pgfmathsetmacro{\theta}{\i*\angle}
       \pgfmathsetmacro{\x}{\rTop*cos(\theta)}
       \pgfmathsetmacro{\y}{\rTop*sin(\theta)}
       \coordinate (T\i) at (\x, \y, \height); % T表示顶部
     }
    
     % 绘制侧棱
     \foreach \i in {0,...,\n} {
       \draw[black, thick] (B\i) -- (T\i);
     }
    
     % 绘制下底面边
     \foreach \i in {0,...,\n} {
       \pgfmathsetmacro{\next}{mod(\i+1,\n+1)}
       \draw[black, thick] (B\i) -- (B\next);
     }
    
     % 绘制上底面边
     \foreach \i in {0,...,\n} {
       \pgfmathsetmacro{\next}{mod(\i+1,\n+1)}
       \draw[black, thick] (T\i) -- (T\next);
     }
    
     % 标记顶点
      \foreach \i in {0,...,\n} {
       \draw[black, line width=0.6pt, fill=black] (B\i) circle (3pt);
       \draw[black, line width=0.6pt, fill=white] (T\i) circle (3pt);
     }
    
     \node[below, align=center] at (0,0,-1.5) {$a=0$};
   \end{scope}
  
   \end{tikzpicture}  
   \newline
 %/////////////////////////////////////////////////////////////////////////////////////////////////////////////  
 \begin{tikzpicture}[tdplot_main_coords, line join=round, line cap=round]
  
   % 全局参数调整，使五个图适合一行
   \def\scaleFactor{0.6}
   \def\spacing{3.0}

   % ====== 图4：同样的棱柱（与图2类似） ======
   \begin{scope}[xshift=-1*\spacing cm, scale=\scaleFactor]
     % 参数设置
     \def\n{20}            % 边数
     \def\angle{18}        % 十边形的中心角
     \def\radius{2.0}     % 底面半径
     \def\height{0.4}      % 高度（较小值实现扁平效果）
    
     % 绘制下底面（十边形）
     \foreach \i in {0,...,\n} {
       \pgfmathsetmacro{\theta}{\i*\angle}
       \pgfmathsetmacro{\x}{\radius*cos(\theta)}
       \pgfmathsetmacro{\y}{\radius*sin(\theta)}
       \coordinate (B\i) at (\x, \y, 0);
     }
    
     % 绘制上底面（十边形）
     \foreach \i in {0,...,\n} {
       \pgfmathsetmacro{\theta}{\i*\angle}
       \pgfmathsetmacro{\x}{\radius*cos(\theta)}
       \pgfmathsetmacro{\y}{\radius*sin(\theta)}
       \coordinate (T\i) at (\x, \y, \height);
     }
    
     % 绘制侧棱
     \foreach \i in {0,...,\n} {
       \draw[black, thick] (B\i) -- (T\i);
     }
    
     % 绘制下底面边
     \foreach \i in {0,...,\n} {
       \pgfmathsetmacro{\next}{mod(\i+1,\n+1)}
       \draw[black, thick] (B\i) -- (B\next);
     }
    
     % 绘制上底面边
     \foreach \i in {0,...,\n} {
       \pgfmathsetmacro{\next}{mod(\i+1,\n+1)}
       \draw[black, thick] (T\i) -- (T\next);
     }
    
     % 标记顶点
      \foreach \i in {0,...,\n} {
       \draw[black, line width=0.6pt, fill=white] (B\i) circle (3pt);
       \draw[black, line width=0.6pt, fill=black] (T\i) circle (3pt);
     }
    
     \node[below, align=center] at (0,0,-1.5) {$a=1$};
   \end{scope}
  
   % ====== 图5：倒置的棱台 ======
     \begin{scope}[xshift=0*\spacing cm, scale=\scaleFactor]
     % 参数设置
     \def\n{20}            % 边数
     \def\angle{18}        % 十边形的中心角
     \def\rBottom{2.0}     % 下底半径
     \def\rTop{1.8}        % 上底半径（稍小，形成棱台）
     \def\height{0.3}      % 高度（较小值实现扁平效果）
    
     % 绘制下底面（十边形）
     \foreach \i in {0,...,\n} {
       \pgfmathsetmacro{\theta}{\i*\angle}
       \pgfmathsetmacro{\x}{\rBottom*cos(\theta)}
       \pgfmathsetmacro{\y}{\rBottom*sin(\theta)}
       \coordinate (B\i) at (\x, \y, 0); % B表示底部
     }
    
     % 绘制上底面（十边形）
     \foreach \i in {0,...,\n} {
       \pgfmathsetmacro{\theta}{\i*\angle}
       \pgfmathsetmacro{\x}{\rTop*cos(\theta)}
       \pgfmathsetmacro{\y}{\rTop*sin(\theta)}
       \coordinate (T\i) at (\x, \y, \height); % T表示顶部
     }
    
     % 绘制侧棱
     \foreach \i in {0,...,\n} {
       \draw[black, thick] (B\i) -- (T\i);
     }
    
     % 绘制下底面边
     \foreach \i in {0,...,\n} {
       \pgfmathsetmacro{\next}{mod(\i+1,\n+1)}
       \draw[black, thick] (B\i) -- (B\next);
     }
    
     % 绘制上底面边
     \foreach \i in {0,...,\n} {
       \pgfmathsetmacro{\next}{mod(\i+1,\n+1)}
       \draw[black, thick] (T\i) -- (T\next);
     }
    
      \foreach \i in {0,...,\n} {
       \draw[black, line width=0.6pt, fill=white] (B\i) circle (3pt);
       \draw[black, line width=0.6pt, fill=black] (T\i) circle (3pt);
     }
    
     \node[below, align=center] at (0,0,-1.5) {$a=2$};
   \end{scope}
     \begin{scope}[xshift=1.*\spacing cm , scale=\scaleFactor]
     % 参数设置
     \def\n{20}            % 边数
     \def\angle{18}        % 十边形的中心角
     \def\rBottom{2.0}     % 下底半径
     \def\rTop{1.5}        % 上底半径（稍小，形成棱台）
     \def\height{0}      % 高度（较小值实现扁平效果）
    
     % 绘制下底面（十边形）
     \foreach \i in {0,...,\n} {
       \pgfmathsetmacro{\theta}{\i*\angle}
       \pgfmathsetmacro{\x}{\rBottom*cos(\theta)}
       \pgfmathsetmacro{\y}{\rBottom*sin(\theta)}
       \coordinate (B\i) at (\x, \y, 0); % B表示底部
     }
    
     % 绘制上底面（十边形）
     \foreach \i in {0,...,\n} {
       \pgfmathsetmacro{\theta}{\i*\angle}
       \pgfmathsetmacro{\x}{\rTop*cos(\theta)}
       \pgfmathsetmacro{\y}{\rTop*sin(\theta)}
       \coordinate (T\i) at (\x, \y, \height); % T表示顶部
     }
    
     % 绘制侧棱
     \foreach \i in {0,...,\n} {
       \draw[black, thick] (B\i) -- (T\i);
     }
    
     % 绘制下底面边
     \foreach \i in {0,...,\n} {
       \pgfmathsetmacro{\next}{mod(\i+1,\n+1)}
       \draw[black, thick] (B\i) -- (B\next);
     }
    
     % 绘制上底面边
     \foreach \i in {0,...,\n} {
       \pgfmathsetmacro{\next}{mod(\i+1,\n+1)}
       \draw[black, thick] (T\i) -- (T\next);
     }
    
     % 标记顶点
      \foreach \i in {0,...,\n} {
       \draw[black, line width=0.6pt, fill=white] (B\i) circle (3pt);
       \draw[black, line width=0.6pt, fill=black] (T\i) circle (3pt);
     }
    
     \node[below, align=center] at (0,0,-1.5) {$a=+\infty$};
   \end{scope}

 \end{tikzpicture}
     \caption{Clips of the movement of two circles under the transform $\Phi_a$ for various $a$. The bubbles of bar circle are denoted by filled dots, while that of hat circle are denoted by hallow dots. }
     \label{fig:2}
 \end{figure}
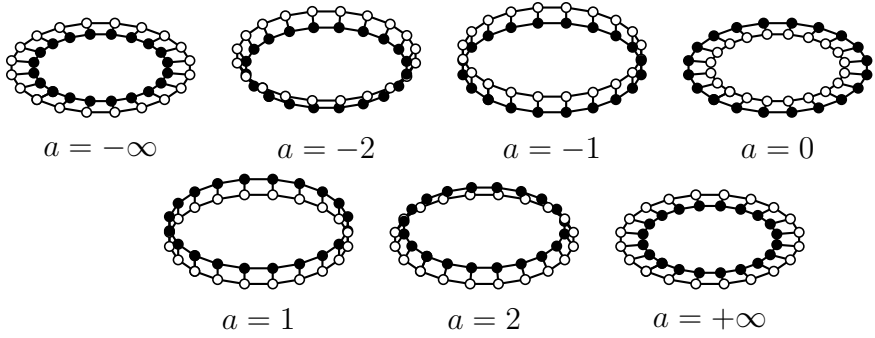

\section{A First Approximation}

From this section, we start to prove the main theorem. We will prove the case $m=1$ in detail and point out the difference and necessary modifications in the remarks for $m=0$. From now on, if there is no further explanation, we shall omit the subscript $m=1$, i.e. denoting $u_{*,1},\mu_1,\lambda_1,r_1,R_1,d_1,\bar\xi_{i,m},\hat\xi_{j,m}$ by $u_*,\mu,\lambda,r,R,d,\bar\xi_{i},\hat\xi_{j}$. 

We begin by performing a first-step approximation, which quantifies the error of the approximation. To analyze the error, we adopt the norms for $\frac{n}{2}<q<n$,
\[||h||_{**}=||(1+|y|)^{n+2-\frac{2n}{q}}h||_{L^q(\mathbb{R}^n)}.\]

From this section until Section 6, we define the following sets
\[\bar B_i:=B(\bar\xi_i,\frac{\delta}k);\ \hat B_i:=B(\hat\xi_i,\frac{\delta}k); \ \text{Int}:=\bigcup_{i=1}^{k}(\bar B_i\cup \hat B_i);\ \text{Ext}:={\mathbb{R}^n}\setminus\text{Int}.\]

We now consider the error of the approximation. The error function is defined as 
\begin{equation}
E:=\Delta u_*+\gamma|u_*|^{p-1}u_*.
\end{equation}

Using equation (1), which can be solved by $U$ and $\bar U_i$, $\hat U_i$, also the linearity of the Laplace operator, we obtain the following formula for $E$ in terms of $U$ and $\bar U_i$, $\hat U_i$, without their derivatives, i.e.,
{\small \begin{align*}
    \gamma^{-1}E:=&\Big(-U^{p}+\sum_{j=1}^{k}\bar U^{p}_j+\sum_{j=1}^{k}\hat U^{p}_j\Big)+\bigg|U-\sum_{j=1}^{k}\bar U_j-\sum_{j=1}^{k}\hat U_j\bigg|^{p-1}\bigg(U-\sum_{j=1}^{k}\bar U_j-\sum_{j=1}^{k}\hat U_j\bigg).
\end{align*}}
Now we investigate the error separately, in the inner domain, and the exterior domain.\\
$\bullet\ $\textit{The inner domain.} For the inner domain, on each $\bar B_i$ and $\hat B_i$, we have 
\[\bar U_j=\frac{\mu^{\frac{n-2}2}}{|x-\bar\xi_j|^{n-2}}+o\bigg(\frac{\mu^{\frac{n-2}2}}{|x-\bar\xi_j|^{n-2}}\bigg);\ \hat U_j=\frac{\lambda^{\frac{n-2}2}}{|x-\hat\xi_j|^{n-2}}+o\bigg(\frac{\lambda^{\frac{n-2}2}}{|x-\hat\xi_j|^{n-2}}\bigg);\ j\neq i.\]
Without loss of generality, we consider $\bar B_1$. The value of $E$ over $\bar B_1$ is dominated by $\bar U_1$, and the rest of the scaled bubbles present as small terms.  By the mean value theorem, we have
\begin{equation*}
\begin{aligned}
\gamma^{-1}E
%\ &\left(-U^{p}+\sum_{j=2}^{k}\bar U^{p}_j+\sum_{j=1}^{k}\hat U^{p}_j\right)+\\
%&+\left|U-\sum_{j=1}^{k}\bar U_j-\sum_{j=1}^{k}\hat U_j\right|^{p-1}\left(\left|U-\sum_{j=1}^{k}\bar U_j-\sum_{j=1}^{k}\hat U_j\right|\right)+\bar U_1^p.\\
=\ &p\left|-U_1+s\left(U-\sum_{j=2}^{k}\bar U_j-\sum_{j=1}^{k}\hat U_j\right)\right|^{p-1}\left(U-\sum_{j=2}^{k}\bar U_j-\sum_{j=1}^{k}\hat U_j\right)-\\
&-U^{p}+\sum_{j=2}^{k}\bar U^{p}_j+\sum_{j=1}^{k}\hat U^{p}_j.
\end{aligned}
\end{equation*}
Now we consider the change of variables, $\gamma^{-1}\tilde E(y):=\gamma^{-1}\mu^{\frac{n+2}{2}}E(\mu y+\bar\xi_1)$
\begin{equation*}
\begin{aligned}
&\gamma^{-1}\tilde E
=\ p\bigg|-U+s\mu^{\frac{n-2}2}\bigg(U(\mu y+\bar\xi_1)-\sum_{j=2}^{k}\bar U_j(\mu y+\bar\xi_1)-\\
&\ -\sum_{j=1}^{k}\hat U_j(\mu y+\bar\xi_1)\bigg)\bigg|^{p-1}\mu^{\frac{n-2}2}\left(U(\mu y+\bar\xi_1)-\sum_{j=2}^{k}\bar U_j(\mu y+\bar\xi_1)-\sum_{j=1}^{k}\hat U_j(\mu y+\bar\xi_1)\right)+\\
&\quad\quad\quad\quad\quad+\mu^{\frac{n+2}2}\bigg(-U^{p}(\mu y+\bar\xi_1)+\sum_{j=2}^{k}\bar U^{p}_j(\mu y+\bar\xi_1)+\sum_{j=1}^{k}\hat U^{p}_j(\mu y+\bar\xi_1)\bigg).\\
\end{aligned}
\end{equation*}
Using the asymptotic main term, we see that
\[|\gamma^{-1}\tilde E|\leq C\left(\frac{\mu^{\frac{n-2}2}}{1+|y|^4}+\mu^{\frac{n+2}{2}}\right).\]
Then by direct calculation, we have  
\[||(1+|y|)^{n+2-\frac{2n}q}\gamma^{-1}\tilde E||_{L^q(B_{(\lambda k)^{-1}\delta})}\leq Ck^{-n/q};\ n\geq4\]
and 
\[||(1+|y|)^{n+2-\frac{2n}q}\gamma^{-1}\tilde E||_{L^q(B_{(\lambda k)^{-1}\delta})}\leq Ck^{-1}\log^{-1}k;\ n=3.\]
The cases of $\hat B_j$ are nothing more than tautology of the previous argument.\\
$\bullet\ $\textit{The outer domain.} Now we turn to the outer domain. In fact, we have the following asymptotic formula on the outer domain for $1\leq j\leq k$,
\[\bar U_j=\frac{\mu^{\frac{n-2}2}}{|x-\bar\xi_j|^{n-2}}+o\left(\frac{\mu^{\frac{n-2}2}}{|x-\bar\xi_j|^{n-2}}\right);\ \hat U_j=\frac{\lambda^{\frac{n-2}2}}{|x-\hat\xi_j|^{n-2}}+o\left(\frac{\lambda^{\frac{n-2}2}}{|x-\hat\xi_j|^{n-2}}\right).\]
By invoking the pointwise estimate $|(a+b)^p-a^p|\leq C(pa^{p-1}b+b^p)$,
%\[|(a+b)^p-a^p|\leq C(pa^{p-1}b+b^p),\]
we have 
\[|\gamma^{-1}E|\leq C\bigg(\bigg(pU^{p-1}+\bigg |\sum_{j=1}^{k}\bar U_j+\sum_{j=1}^{k}\hat U_j\bigg|^{p-1}\bigg)\bigg(\sum_{j=1}^{k}\bar U_j+\sum_{j=1}^{k}\hat U_j\bigg)+\sum_{j=1}^{k}\bar U_j^p+\sum_{j=1}^{k}\hat U_j^p\bigg).\]
Moreover, the last two sums are small terms with respect to the first. Thus, we have
\begin{equation*}
\begin{aligned}
|\gamma^{-1}E|\leq & \ C\left(\frac1{(1+y^2)^2}+\left|\sum_{j=1}^k\left(\frac{\mu^{\frac{n-2}2}}{|y-\bar\xi_j|^{n-2}}+\frac{\lambda^{\frac{n-2}2}}{|y-\hat\xi_j|^{n-2}}\right)\right|^{\frac4{n-2}}\right)\times\\
& \qquad\qquad\qquad\times\left(\sum_{j=1}^k\left(\frac{\mu^{\frac{n-2}2}}{|y-\bar\xi_j|^{n-2}}+\frac{\lambda^{\frac{n-2}2}}{|y-\hat\xi_j|^{n-2}}\right)\right)\\
\leq&\ C\frac1{(1+y^2)^2}\sum_{j=1}^k\left(\frac{\mu^{\frac{n-2}2}}{|y-\bar\xi_j|^{n-2}}+\frac{\lambda^{\frac{n-2}2}}{|y-\hat\xi_j|^{n-2}}\right)\\
\end{aligned}
\end{equation*}
Now we are ready to integrate both sides of the inequality. Using the triangular inequality of norms, we have
\begin{align*}
&||(1+|y|)^{n+2-\frac{2n}q}E||_{L^{q}(\text{Ext})}\\
\leq\ &C\left\|\frac{(1+|y|)^{n+2-\frac{2n}q }}{(1+y^2)^2}\sum_{j=1}^k\left(\frac{\mu^{\frac{n-2}2}}{|y-\bar\xi_j|^{n-2}}+\frac{\lambda^{\frac{n-2}2}}{|y-\hat\xi_j|^{n-2}}\right)\right\|_{L^{q}(\text{Ext})}\\
\leq\ & C\sum_{j=1}^k\left\|\frac{(1+|y|)^{n+2-\frac{2n}q }}{(1+y^2)^2}\left(\frac{\mu^{\frac{n-2}2}}{|y-\bar\xi_j|^{n-2}}+\frac{\lambda^{\frac{n-2}2}}{|y-\hat\xi_j|^{n-2}}\right)\right\|_{L^{q}(\text{Ext})}.\\
\leq\ & C\left(\mu^{\frac{n-2}2}\sum_{j=1}^k\left\|\frac{(1+|y|)^{n-2-\frac{2n}q}}{|y-\bar\xi_j|^{n-2}}\right\|_{L^{q}(\bar B_{i}^C)}+\lambda^{\frac{n-2}2}\sum_{j=1}^k\left\|\frac{(1+|y|)^{n-2-\frac{2n}q}}{|y-\hat\xi_j|^{n-2}}\right\|_{L^{q}(\hat B_{i}^C)}\right)\\
\leq\ & C\Bigg(\mu^{\frac{n-2}2}k\left(\int_{\frac\delta k}^{1}\frac{s^{n-1}}{s^{(n-2)q}}\right)^{\frac1q}+\lambda^{\frac{n-2}2}k\left(\int_{\frac\delta k}^{1}\frac{s^{n-1}}{s^{(n-2)q}}\right)^{\frac1q}\Bigg)\leq Ck^{1-\frac nq}.
\end{align*}
Repeating the proof for the case in dimension three, we have 
\[||(1+|y|)^{n+2-\frac{2n}q}E||_{L^{q}(\text{Ext})}\leq C\log^{-1}k.\]

So far, we have finished the first-step approximation. By the previous calculation, the error turns out to be
\begin{equation}
    ||E||_{**}\leq Ck^{1-\frac nq},\ n\geq4;\ ||E||_{**}\leq C\log^{-1}k,\ n=3.
\end{equation}

\begin{remark}
    The calculation of first step approximation for $u_{0}$ by $u_{*,0}$ is the same as this section, the only change is that the inner and outer domains are adapted to the corresponding configurations, that is to consider the function $u_{*,0}$ on the sets
    \[\bar B_{i,0}:=B(\bar\xi_{i,0},\frac{\delta}k);\ \hat B_{i,0}:=B(\hat\xi_{i,0},\frac{\delta}k); \ \text{Int}_0:=\bigcup_{i=1}^{k}(\bar B_{i,0}\cup \hat B_{i,0});\ \text{Ext}_0:={\mathbb{R}^n}\setminus\text{Int}_0,\]
    and to repeat the process performed above.
    In fact, the results of the estimates are also identical. 
\end{remark}
%%%%%%%%%%%%%%%%%%%%%%%%%%%%%%%%%%%%%%%%%%%%%%%%%%%%%%%

\section{Gluing Scheme}
From now on, we denote the weighed infinity-norm
\[||g||_{*}:=||(1+|y|^{n-2})g||_{L^\infty(\mathbb{R}^n).}\]
In this section, we will follow the inner-outer gluing scheme, and solve the inner and outer problems respectively.\par

Our aim is to figure out the aforementioned error term of the approximate solution $u_*$, say $\phi$. That is, the function $u=u_*+\phi$ solves the problem in the space, then the function $\phi$ satisfies
\begin{equation}
    \Delta \phi+p\gamma|u_*|^{p-1}\phi+E+\gamma N(\phi)=0,
\end{equation}
with 
\[N(\phi):=|u|^{p-1}u-|u_*|^{p-1}u_*-p|u_*|^{p-1}\phi.\]\par
Let $\zeta(s)$ be a smooth decreasing function, with $\zeta(s)=1$ for $s<1$ and $\zeta(s)=0$ for $s>2$. Then we define 
\begin{align*}
    \hat\zeta_j(y)=\zeta{\left(\frac k{\delta}\left|\frac{y}{|y|^2}-\hat\xi_j\right|\right)};\quad \bar\zeta_j(y)=\zeta\left(\frac{k}{\delta}|y-\bar\xi_j|\right),\quad 1\leq j\leq k.
\end{align*}
% \begin{equation*}
% \left\{
% \begin{aligned}
%     \hat\zeta_j(y)=\zeta{\left(\frac k{\delta}\left|\frac{y}{|y|^2}-\hat\xi_j\right|\right)}&;\\
%     \bar\zeta_j(y)=\zeta\left(\frac{k}{\delta}|y-\bar\xi_j|\right)\ \ \ \ \ \ &;
% \end{aligned}
% \right.\ \text{for}\ j=1,...,k,
% \end{equation*}
It is easy to see that $\hat\zeta_j(y)=\mathcal{T}\bar\zeta_j(y/|y|^2)$.

We now consider such $\phi$ of the form
\[\phi=\sum_{j=1}^k(\bar\phi_j+\hat\phi_j)+\psi.\]
Then the inner and outer problems spell as follows
\begin{equation}\label{innereqbar}
    \Delta\bar\phi_j+p\gamma|u_*|^{p-1}\bar\zeta_j\bar\phi_j+\bar\zeta_j(p\gamma|u_*|^{p-1}\psi+E+\gamma N(\phi))=0,\ j=1,...,k;
\end{equation}
\begin{equation}\label{innereqhat}
    \Delta\hat\phi_j+p\gamma|u_*|^{p-1}\hat\zeta_j\hat\phi_j+\hat\zeta_j(p\gamma|u_*|^{p-1}\psi+E+\gamma N(\phi))=0,\ j=1,...,k;
\end{equation}
\begin{equation}\label{outereq}
\begin{aligned}
    &\Delta\psi+p\gamma U^{p-1}\psi+p\gamma\left(\left|u_*\right|^{p-1}-U^{p-1}\right)\left(1-\sum_{j=1}^{k}\left(\bar\zeta_j+\hat\zeta_j\right)\right)\psi\\
    &-p\gamma U^{p-1}\sum_{j=1}^{k}\left(\bar\zeta_j+\hat\zeta_j\right)\psi+p\gamma|u_*|^{p-1}\sum_{j=1}^k\left(1-\bar\zeta_j\right)\bar\phi_j+\\
    &+p\gamma|u_*|^{p-1}\sum_{j=1}^k\left(1-\hat\zeta_j\right)\hat\phi_j+\left((1-\sum_{j=1}^{k}\left(\bar\zeta_j+\hat\zeta_j\right)\right)\times\\
    &\qquad\qquad\qquad\qquad\qquad\qquad\qquad\qquad\times\left(E+\gamma N(\phi)\right)=0
\end{aligned}
\end{equation}

The first $2k$ equations are called inner problems, and the last one is called outer problem. In order to maintain the symmetry of $\phi$, we need to impose the following symmetry condition to the solutions of the inner problems.
\begin{equation}
    \bar\phi_j=\mathcal{T}^{2j-2}\bar\phi_1,\ j=2,...,k;
\end{equation}
\begin{equation}
\bar\phi_1(y_1,...,y_i,...,y_n)=\bar\phi_1(y_1,...,-y_i,...,y_n),\ 3\leq i\leq n.
\end{equation}
To obtain some analogy of geometric ``congruence" condition, we would like to connect the inner circles and the outer circle, namely.
\begin{equation}
    \hat\phi_1=\mathcal K\bar\phi_1.
\end{equation}

\begin{proposition}
There exist constants $k_{0}$, $C$ and $\rho_{0}$ such that, for all integer $k \geq k_{0}$ and positive real number $\rho < \rho_{0}$, if $\bar{\phi}_{j}$ and $\hat{\phi}_{j}$ satisfy the conditions (17)-(19), then there exists a unique solution $\psi = \Psi(\bar{\phi}_{1},\hat{\phi}_{1})$ of the inner problem (16) such that
\begin{equation}
    \psi=\mathcal{T}^{2j}\psi,\ j=1,...,k;
\end{equation}
\begin{equation}
\psi(y_1,...,y_i,...,y_n)=\psi(y_1,...,-y_i,...,y_n),\ 3\leq i\leq n;
\end{equation}
\begin{equation}
    \psi=\mathcal  K\psi
\end{equation}
and
\begin{equation}
\begin{aligned}
    &\ \|\psi\|_{*} \leq C\left(\|\bar{\bar{\phi}}_{1}\|_{*} +||\hat{\hat\phi}_1||_* + k^{1-\frac{n}{q}}\right)\  \text{if}\  n \geq 4, \\
    &\|\psi\|_{*} \leq C\left(\|\bar{\bar{\phi}}_{1}\|_{*}+ ||\hat{\hat\phi}_1||_* + \log^{-1}k\right) \ \text{if} \ n = 3.
\end{aligned}
\end{equation}
where $\bar{\bar\phi}_1=\mu^{\frac{n-2}2}\bar{\bar\phi}(\mu y+\bar\xi_1)$, $\hat{\hat\phi}=\lambda^{\frac{n-2}2}\hat\phi(\lambda x+\hat \xi_1)$.
Furthermore, the operator $\Psi$ satisfies
\[
\left\|\Psi(\bar{\bar\phi}_{1},\hat{\hat\phi}_1) - \Psi(\bar{\bar\phi}_{2},\hat{\hat{\phi}}_2)\right\|_{*} \leq C \big(\left\|\bar{\bar\phi}^{1} - \bar{\bar\phi}^{2}\right\|_{*}+\left\|\hat{\hat\phi}^{1} - \hat{\hat\phi}^{2}\right\|_{*}\big).
\]
\end{proposition}

\noindent\textbf{Proof.} We proceed via classical fixed point theory. To apply this method, we need two steps.\\
\textit{Step 1.} Linear theory. We start the proof by considering the linear equation
\begin{equation}
    \Delta \psi+p\gamma U^{p-1}\psi=h.
\end{equation}
with the function $h$ satisfying the symmetry condition (20), (21) as well as 
\[h(y)=|y|^{-n-2}h(\mathcal{T}y/|y|^2).\]
The linear theory has been well studied, see \cite{DELPINO2011}, Lemma 3.1. By invoking the linear theory here, all we need to check is the orthogonality with respect to the basis vectors of the kernel. For $j=1,...,n$, repeating the process as Lemma 4.1 in \cite{DELPINO2011} will lead to $\int hZ_j=0$. For $j=n+1$, the case is slightly different. Thus, we define 
\[I(\lambda)=\lambda^{\frac{n-2}{2}}\int h(y)U(\lambda y).\]
Making a substitute $y\mapsto \mathcal{T}y|y|^{-2}$, we see that $I(\lambda)=I(\lambda^{-1})$, hence \[I'(1)=\int h(y)Z_{n+1}(y)=0.\]
Invoking the linear theory and the symmetry of the solution, the uniqueness and existence are clear. From now on, we denote the (linear) solution operator as $T$, then by \cite{DELPINO2011}, Lemma 3.1, we also have boundedness of $T$, i.e.,
\[||T(h)||:=||\psi||_*\leq C||h||_{**}.\]
\textit{Step 2.} Fixed point argument.\par
We denote the ``potential" $V=V_1+V_2$, where
\[V_1(y):=p\gamma(|u_*|^{p-1}-U^{p-1})(1-\sum_{j=1}^k(\bar\zeta_j+\hat\zeta_j));\]
and
\[V_2(y)=p\gamma U^{p-1}\sum_{j=1}^k(\bar\zeta_j+\hat\zeta_j),\]
also denote
\[M(\psi)=(1-\sum_{j=1}^k(\bar\zeta_j+\hat\zeta_j))(E+\gamma N(\phi)).\]
Our fixed point problem turns out to be
\begin{equation}
\psi=-T\left(V\psi+p\gamma|u_*|^{p-1}\sum_{j=1}^k((1-\bar\zeta_j)\bar\phi_j+(1-\hat\zeta_j)\hat\phi_j)+M(\psi)\right)=\mathcal M(\psi).
\end{equation}
Since $T$, the solution operator, is a bounded linear operator, then we only need to check that the right-hand side term is a contracting map.\par
For $V_1$, we have 
\[V_1=C\left|U-s\sum_{j=1}^k(\bar U_j+\hat U_j)\right|^{p-2}\sum_{j=1}^k(\bar U_j+\hat U_j)(1-\sum_{j=1}(\bar\zeta_j+\hat\zeta_j))\]
Then using the estimates on Ext, we have
\[|V_1|\leq CU^{p-2}\sum_{i=1}^k\left(\frac{\mu^{\frac{n-2}2}}{|y-\bar\xi_i|^{n-2}}+\frac{\lambda^{\frac{n-2}2}}{|y-\hat\xi_i|^{n-2}}\right).\]
then
\[|V_1\psi|\leq C||\psi||_*U^{p-1}\sum_{i=1}^k\left(\frac{\mu^{\frac{n-2}2}}{|y-\bar\xi_i|^{n-2}}+\frac{\lambda^{\frac{n-2}2}}{|y-\hat\xi_i|^{n-2}}\right).\]
As was shown in section 3, we can calculate that
$$||V_1\psi||_{**}\leq \frac{C||\psi||_*}{k^{\frac nq-1}},\ n\geq4;\ ||V_1\psi||_{**}\leq \frac{C||\psi||_*}{\log k},\ n=3.$$
Also, direct calculation gives
\[||V_2\psi||_{**}\leq C\frac{||\psi||_*}{k^{\frac nq-1}}.\]
The second term can be estimated as follows
\[\left||u_*|^{p-1}\sum_{j=1}^k((1-\bar\zeta_j)\bar\phi_j+(1-\hat\zeta_j)\hat\phi_j)\right|\leq CU^{p-1}\sum_{j=1}^k(|\bar\phi_j|+|\hat\phi_j|),\]
we also have obtained the pointwise estimates for $\bar\phi_i$ and $\hat \phi_i$, by the definition of $\bar{\bar\phi}$, 
\[|\bar\phi_j|\leq C||\bar{\bar\phi}||_*\frac{\mu^{\frac{n-2}{2}}}{\mu^{n-2}+|y-\bar\xi_j|^{n-2}};\ |\hat\phi_j|\leq C||\hat{\hat\phi}||_*\frac{\lambda^{\frac{n-2}{2}}}{\lambda^{n-2}+|y-\hat\xi_j|^{n-2}}.\]
Combining these three estimates gives
\begin{equation}
\left\||u_*|^{p-1}\sum_{j=1}^k((1-\bar\zeta_j)\bar\phi_j+(1-\hat\zeta_j)\hat\phi_j)\right\|_{**}\leq
\left\{
\begin{aligned}
&Ck^{1-\frac{n}{q}}(||\bar{\bar\phi}_1||_{*}+||\hat{\hat{\phi}}_1||_*)\ \ n\geq4;\\
&C\log^{-1}k(||\bar{\bar\phi}_1||_*+||\hat{\hat{\phi}}_1||_*)\ \ n=3.
\end{aligned}
\right.
\end{equation}

Finally, using the norm estimate of $E$ on Ext, the pointwise estimates for $\hat\phi_i$ and the following pointwise estimate for the case in which $|\phi|\ll |u_*|$,
\[||u|^{p-1}u-|u_*|^{p-1}u_*-p|u_{*}|^{p-1}\phi|\leq C|u_*|^{p-2}\phi^2,\]
we have
\begin{equation}
||M(\psi)||_{**}\leq\left\{
\begin{aligned}
\frac C{k^{\frac nq-1}}(1+||\bar{\bar\phi}_1||_*^2+||\hat{\hat\phi}_1||_*^2)+C||\psi||_*^2,\ &\ n\geq4;\\
\frac C{\log k}(1+||\bar{\bar\phi}_1||_*^2+||\hat{\hat\phi}_1||_*^2)+C||\psi||_*^2,\ &\ n=3.
\end{aligned}
\right.
\end{equation}
If $|\psi_1|,|\psi_2|\leq\rho$, a tautological proof gives
\begin{equation}
    ||M(\psi_1)-M(\psi_2)||_{**}\leq C\rho||\psi_1-\psi_2||_*.
\end{equation}
Combining previous calculation, also the fact that $T$ is bounded, we see that $\mathcal M$ is a contracting map. Also, we have, from the fixed point problem (19) and the estimates above,
\begin{equation*}
\begin{aligned}
    &\ \|\psi\|_{*} \leq C\left(\|\bar{\bar{\phi}}_{1}\|_{*} +||\hat{\hat\phi}_1||_* + k^{1-\frac{n}{q}}\right)\  \text{if}\  n \geq 4, \\
    &\|\psi\|_{*} \leq C\left(\|\bar{\bar{\phi}}_{1}\|_{*}+ ||\hat{\hat\phi}_1||_* + \log^{-1}k\right) \ \text{if} \ n = 3,
\end{aligned}
\end{equation*}
which concludes the proof of the proposition.\qed

Having solved the outer problem, we now consider the solution $\psi$ as $\psi=\Psi(\bar{\bar\phi}_1,\hat{\hat\phi}_1)$, and the inner problems become
\begin{equation*}
    \Delta\bar\phi_1+p\gamma|u_*|^{p-1}\bar\zeta_1\bar\phi_1+\bar\zeta_1(p\gamma|u_*|^{p-1}\Psi(\bar{\bar\phi}_1,\hat{\hat\phi}_1)+E+\gamma N(\phi))=0,
\end{equation*}
\begin{equation*}
    \Delta\hat\phi_1+p\gamma|u_*|^{p-1}\hat\zeta_1\hat\phi_1+\hat\zeta_1(p\gamma|u_*|^{p-1}\Psi(\bar{\bar\phi}_1,\hat{\hat\phi}_1)+E+\gamma N(\phi))=0,
\end{equation*}
To simplify the notations, we rewrite them as
\begin{equation}
    \Delta\bar\phi_1+p\gamma|\bar U_1|^{p-1}\bar\zeta_1\bar\phi_1+\bar\zeta_1E+\gamma\bar{\mathcal{N}}(\bar\phi_1,\hat\phi_1,\phi)=0,
\end{equation}
\begin{equation}
    \Delta\hat\phi_1+p\gamma|\hat U_1|^{p-1}\hat\zeta_1\hat\phi_1+\hat\zeta_1E+\gamma\hat{\mathcal{N}}(\bar\phi_1,\hat\phi_1,\phi)=0;
\end{equation}
by introducing 
\[\bar{\mathcal{N}}(\bar\phi_1,\hat\phi_1,\phi):=p(|u_*|^{p-1}\bar\zeta_1-|\bar U_1|^{p-1})\bar\phi_1+\bar\zeta_1(p|u_*|^{p-1}\Psi(\bar{\bar{\phi}}_1,\hat{\hat\phi}_1)+N(\phi)),\]
\[\hat{\mathcal{N}}(\bar\phi_1,\hat\phi_1,\phi):=p(|u_*|^{p-1}\bar\zeta_1-|\hat U_1|^{p-1})\hat\phi_1+\hat\zeta_1(p|u_*|^{p-1}\Psi(\bar{\bar{\phi}}_1,\hat{\hat\phi}_1)+N(\phi)).\]
This equation is solved by the same method. If the solvable conditions, i.e., the orthogonal conditions, are satisfied, then we are done. We may first solve equations under the assumption that
\begin{equation}
    \bar c_{1}=\frac{\int_{\mathbb{R}^n}(\bar\zeta_1E+\gamma\bar{\mathcal{N}}(\bar\phi_1,\hat\phi_1,\phi))\bar Z_{1}}{\int_{\mathbb{R}^n}\bar U^{p-1}_1\bar Z_{1}^2}=0;\ 
\end{equation}
\begin{equation}
    \bar c_{n+1}=\frac{\int_{\mathbb{R}^n}(\bar\zeta_1E+\gamma\bar{\mathcal{N}}(\bar\phi_1,\hat\phi_1,\phi))\bar Z_{n+1}}{\int_{\mathbb{R}^n}\bar U^{p-1}_1\bar Z_{n+1}^2}=0.
\end{equation}
Since the orthogonality with $\bar Z_{1}$ and $\bar Z_{n+1}$ is non-trivial, which will be thoroughly computed in section 6.
\begin{proposition}
There exists the unique solution $\bar\phi_1$ to the inner problems that satisfies
\[\|\bar{\bar\phi}_1\|_*\leq Ck^{-\frac nq},\ n\geq4;\ \|\bar{\bar\phi}_1\|_*\leq Ck^{-1}\log^{-1}k,\ n=3;\]
and 
\[||\bar{\bar{\mathcal{N}}}(\bar\phi_1,\hat\phi_1,\phi)||_{**}\leq Ck^{-\frac{2n}q},\ n\geq4;\ ||\bar{\bar{\mathcal{N}}}(\bar\phi_1,\hat\phi_1,\phi)||_{**}\leq Ck^{-2}\log^{-2}k,\ n=3.\]
where 
\[\bar{\bar{\mathcal{N}}}(\bar\phi_1,\hat\phi_1,\phi)=\mu^{\frac{n+2}2}\bar{{\mathcal{N}}}(\bar\phi_1,\hat\phi_1,\phi)(\bar\xi_1+\mu y);\]
\[\hat{\hat{\mathcal{N}}}(\bar\phi_1,\hat\phi_1,\phi)=\mu^{\frac{n+2}2}\hat{{\mathcal{N}}}(\bar\phi_1,\hat\phi_1,\phi)(\hat\xi_1+\lambda y).\]
\end{proposition}
\textbf{Proof.} We also divide the proof into two parts.\\
\textit{Step 1.} Linear theory.
Analogously, we invoke the lemma to solve the following linear equation,
\begin{equation}
    \Delta \bar\phi_1+p\gamma \bar U^{p-1}_1\bar\phi_1+\bar h=0 \ \text{in } \mathbb{R}^n;
\end{equation}
\begin{equation}
    \Delta \hat\phi_1+p\gamma \hat U^{p-1}_1\hat\phi_1+\hat h=0 \ \text{in } \mathbb{R}^n;
\end{equation}
where $\bar h,\hat h$ satisfies the symmetries (11)-(13). Also we impose the condition that
\[\hat h(y)=\mathcal K(\bar h)(y)=|y|^{-2-n}\bar h(\mathcal Ty/|y|^2).\]
That is, $\hat h=\mathcal K\bar h$.
\par
By the oddness of $\bar Z_i, 2\leq i\leq n$, we see that $\int_{\mathbb{R}^n}Z_i\bar h=0.$ We only need to check that these two equations are coupled.
In fact, under the assumptions that $c_1,c_{n+1}=0$, we see that the problem (34) has a solution $\bar\phi_1$. Applying the linear transform $\mathcal K$ on both sides of the equation, we have
\[|y|^4\Delta (\mathcal K\bar\phi_1)+p\gamma|y|^4 (\mathcal K\bar U_1)^{p-1}(\mathcal K\bar\phi_1)+|y|^{2-n}h(\mathcal Ty/|y|^2)=0.\]
If the second equation satisfies the condition of congruence, then we have, as was assumed, $\hat h=|y|^{-2-n}\bar h(\mathcal Ty/|y|^2)$, then the function $\mathcal K \bar\phi_1$ solves the second equation, vice versa. That is, the first equation has the unique solution iff the second one has the unique solution. Furthermore, the solution satisfies the norm estimate
\begin{equation}
||\bar{\bar\phi}_1||_{*}\leq C||\bar h||_{**};\ ||\hat{\hat\phi}_1||_{*}\leq C||\hat h||_{**}.
\end{equation}
From now on, we denote the solution operator as $T=(\bar T,\hat T)^{T}$.
\\
\textit{Step 2.} Fixed point argument.\par
We consider the space $X\subset\mathcal D^{1,2}(\mathbb R^n)\times\mathcal D^{1,2}(\mathbb R^n)$, where
\[X=\{(u,v)^T:\mathcal K u=v;\ u,v\in\mathcal D^{1,2}(\mathbb R^n),\ ||u||_{*}+||v||_{*}<\infty\}.\]
The fixed point problem is written as
\begin{equation}
\phi_1:=
\left(
\begin{aligned}
&\bar\phi_1\\
&\hat\phi_1
\end{aligned}
\right)
=
\left(
\begin{aligned}
&\bar T(\bar\zeta_1E+\gamma\bar{\mathcal N})\\
&\hat T(\hat\zeta_1E+\gamma\hat{\mathcal N})
\end{aligned}
\right)
=:
\left(
\begin{aligned}
&\bar{\mathcal M}(\bar\phi_1)\\
&\hat{\mathcal M}(\hat\phi_1)
\end{aligned} 
\right)
=:\mathcal M(\phi_1)
\end{equation}

We thus focus on the first equation. Now we define
\[f_1:=p\bar\zeta_1(|u_*|^{p-1}-U_1^{p-1})\bar\phi_1,\ f_2:=(\bar\zeta_1-1)\bar U_1^{p-1}\bar\phi_1,\]
\[f_3:=p\bar\zeta_1|u_*|^{p-1}\Psi(\bar\phi_1),\ f_4:=\bar\zeta_1N(\phi),\ f_5:=\bar\zeta_1 E,\]
then we deal with each $f_i$ respectively.\par
From now on, we denote $\tilde f(y)=C\lambda^{\frac{n+2}2}f(\bar\xi_1+\mu y)$.
For $f_1$, we now consider $\tilde f_1$. By the pointwise estimate around $\bar \xi_1$, i.e.,
\[\left||u_*|^{p-1}-\bar U_1^{p-1}\right|\leq CU_1^{p-2}|u_*-\bar U_1|\text{ in }\bar B_1,\]
using the summation, we have
\begin{equation}
|\tilde f_1|\leq C\mu^{\frac{n-2}2}U^{p-2}|\bar{\bar\phi}_1|\leq C\mu^{\frac{n-2}2}U^{p-1}||\bar{\bar\phi}_1||_{*}\text{ in } B(0,2(k\mu)^{-1}\delta).
\end{equation}
Then an analogous calculation as in Section 4, we have
\[\|\tilde f_1\|_{**}\leq C \mu^{\frac n{2q}}||\bar{\bar\phi}_1||_*.\]
For $\tilde f_2$, we have 
\begin{equation}
|\tilde f_2|\leq CU^p||\bar{\bar\phi}_1||_*,
\end{equation}
then similarly,
\[\|\tilde f_2\|_{**}\leq C \mu^{\frac n{2q}}||\bar{\bar\phi}_1||_*.\]
As for $f_3$, we have calculated in Propostion 5 that
\[\|\psi\|_{*} \leq C\left(\|\bar{\bar{\phi}}_{1}\|_{*} +||\hat{\hat\phi}_1||_* + k^{1-\frac{n}{q}}\right),\]
then
\[\|\tilde f_3\|_{**} \leq C\mu^{\frac n{2q}}\left(\|\bar{\bar{\phi}}_{1}\|_{*} +||\hat{\hat\phi}_1||_* + k^{1-\frac{n}{q}}\right).\]
For $f_4$, directly writing down the main term
\[|\tilde f_4(y)|\leq CU^{p-2}(y)(|\bar{\bar \phi}_1(y)|^2+\mu^{n-2}|\psi(\bar\xi_1+\mu y)|^2),\ |y|\leq \frac{\delta}{\mu k},\]
and integrating on both sides gives
\[||\tilde f_4||_{**}\leq C\mu^{\frac n{2q}}(\|\bar{\bar\phi}_1\|_*+k^{1-\frac nq})\]
And for $f_5$, directly applying the result of Section 4, we have 
\[||\tilde f_5||_{**}\leq C\mu^{\frac n{2q}}.\]
In summary, we have 
\[\|\bar{\bar{\mathcal N}}\|_{**} \leq C\mu^{\frac n{2q}}\left(\|\bar{\bar{\phi}}_{1}\|_{*} +||\hat{\hat\phi}_1||_* + k^{1-\frac{n}{q}}\right).\]
Meanwhile, for the second equation, the process is almost the same. Computation shows 
\[\|\bar{\bar{\mathcal N}}\|_{**} \leq C\lambda^{\frac n{2q}}\left(\|\bar{\bar{\phi}}_{1}\|_{*} +||\hat{\hat\phi}_1||_* + k^{1-\frac{n}{q}}\right).\]
Moreover, we can also show that the map is contracting on a sufficiently small ball at the origin by a tautological proof, from which we also have, combining the previous calculation,
\[||\bar{\bar\phi}_1||_{*}\leq Ck^{-\frac{2n}{q}}, ||\hat{\hat\phi}_1||_*\leq Ck^{-\frac{2n}{q}}.\]
Repeating the process for dimension 3, we also have the following results,
\[||\bar{\bar\phi}_1||_{*}\leq Ck^{-1}\log^{-1}k;\ ||\bar{\bar{\mathcal{N}}}(\bar\phi_1,\hat\phi_1,\phi)||_{**}\leq Ck^{-2}\log^{-2}k,\]
which finishes the proof.\qed
\begin{remark}
To prove the case corresponding to $m=0$, one needs to replace the parameter of $u_{*,1}$, then repeating the proof of Proposition 6 and 7, to give the outer and inner solutions $\bar\phi_{0,1},\hat\phi_{0,1},\psi_{0}$. The only two things that deserves attention are that the condition (11) should be replaced with $\bar\psi_{0}=K\bar\psi_{0}$, i.e., the authentic Kelvin transform, and the relationship of $\hat h$ and $\bar h$ in the first step in the proof of Proposition 8 is adapted to  $$\hat h(y)=K(\bar h)(y)=|y|^{-2-n}\bar h(y/|y|^2),$$ since $u_{*,0}$ is Kelvin invariant.
\end{remark}
\section{Symmetry Decomposition}
Prior to the final reduction, it is necessary to present the symmetry decomposition of $E$, $\bar\phi_1$, as well as $\psi$. This step is essential for computing $\bar c_1$,  as the existing symmetry alone is insufficient to directly derive the required formula.
In this context, the following asymptotic expansion up to the linear term will also be useful. 
\begin{equation}
\begin{aligned}
{\bar U}_j(\mu y+\bar \xi_1)=\frac{(2\mu)^\frac{n-2}{2}}{|\bar\xi_1-\bar\xi_j|^{n-2}}\left(1-\frac{(n-2)(y,\bar\xi_1-\bar\xi_j)}{|\bar\xi_1-\bar\xi_j|^2}\mu+O\left(\frac{\mu^2(1+|y|^2)}{|\bar\xi_1-\bar\xi_j|^2}\right)\right);\\
{\hat U}_j(\mu y+\bar \xi_1)=\frac{(2\lambda)^\frac{n-2}{2}}{|\bar\xi_1-\hat\xi_j|^{n-2}}\left(1-\frac{(n-2)(y,\bar\xi_1-\hat\xi_j)}{|\bar\xi_1-\hat\xi_j|^2}\mu+O\left(\frac{\mu^2(1+|y|^2)}{|\bar\xi_1-\hat\xi_j|^2}\right)\right);\\
U(\mu y+\bar\xi_1)=U(\bar\xi_1)\left(1-(n-2)\frac{(y,\bar\xi_1)}{1+|\bar\xi_1|^2}\mu+O\left(\frac{\mu^2|y|^2}{1+|\xi_1|^2}\right)\right).\ \ \ \ \ \ \ \ \ \ 
\end{aligned}    
\end{equation}
We see that the constant term is, of course, even with respect to $y_i$.
The following are the symmetry decomposition theorems. Indeed, their proofs follow almost the same arguments as those for Propositions 2.2, Proposition 4.7,4.8,4.9 in \cite{MEDINA2021}.
\begin{proposition}
For the error operator, we have the following decomposition
\[E(\bar\xi_1+\mu y)=E^s(\bar\xi_1+\mu y)+E^*(\bar\xi_1+\mu y).\]
where $E^s(\bar\xi_1+\mu y)$ is even with respect to $y_i$, and $E^*(\bar\xi_1+\mu y)$ satisfies, for sufficiently large $k$,
\begin{equation}
|\mu^{\frac{n+2}2}E^*(\bar\xi_1+\mu y)|\leq\left\{
\begin{aligned}
    C\frac{\mu^{\frac{n-2}2}}{k(1+|y|^3)},\ \ \ \ \ \ \ \ \ \ \ &n\geq4;\\
    C\frac{\mu^{\frac{1}2}}{k\log^3k(1+|y|^3)},\ &n=3;
\end{aligned}
\right.
\ \ \ \ \ |y|<\frac{\delta}{\mu k}.
\end{equation}
\end{proposition}
In fact, the symmetric part $E^s$ is given by replacing $E$ with its constant part in the asymptotic expansion, and then we denote the rest of the part as $E^*$. 

\begin{proposition}
For the inner solution $\bar\phi_1$, we have the following pointwise estimate
\begin{equation}
|\bar{\bar\phi}_1(y)|\leq C\frac{\mu^{\frac{n-2}2}}{k(1+|y|)^a},\ a=\left\{
\begin{aligned}
2,\ n\geq5;\\
1,\ n=4;\\
\alpha,\ n=3,
\end{aligned}
\right.
\end{equation}
where $\alpha\in(0,1)$. We also have the following decomposition for $\bar{\bar\phi}_1$
\[\bar{\bar\phi}_1=\bar{\bar\phi}_1^s+\bar{\bar\phi}_1^*.\]
where $\bar{\bar\phi}_1^s$ is even with respect to $y_i$, and $\bar{\bar\phi}_1^*$ satisfies, for sufficiently large $k$,
\begin{equation}
|\bar{\bar\phi}_1^*(y)|\leq\left\{
\begin{aligned}
    C\frac{\mu^{\frac{n-2}2}}{k(1+|y|)},\ \ \ \ \ \ \ \ \ \ \ \ \ &n\geq4;\\
    C\frac{\mu^{\frac{1}2}}{k\log^3k(1+|y|^\alpha)},\ &n=3,\alpha\in(0,1);
\end{aligned}
\right.
\end{equation}
\end{proposition}
The proofs for the two propositions is a simple repeat of that proven in \cite{MEDINA2021}. However, the following is different, due to the essential differences between our configuration and that of \cite{MEDINA2021}.
\begin{proposition}
    For the outer solution $\psi$, we have the following decomposition for ${\psi}$
\[\psi(\bar\xi_1+\mu y)=\psi^s(y)+\psi^*(y),\ \ y\in B(0,\frac\delta{\mu k})\]
where $\psi^s$ is even with respect to $y_1$, and $\psi^*$ satisfies
\begin{equation*}
|\psi^*(y)|\leq C\left(||\bar{\bar\phi}_1||_*+||\psi||_*+do_{k}(1)\right)\mu|y|(1+|y|).
\end{equation*}
\end{proposition}
\begin{proof}
We consider the inner problem \eqref{outereq}. Using the Newton potential, the function $\psi(\bar\xi_1+\mu y)$ can be written as
\begin{equation}
    \psi(\bar\xi_1+\mu y)=\int |\mu y+\bar\xi_i-x|^{2-n}W(\psi)(x)dx,
\end{equation}
where 
\[W(\psi)=p\gamma|u_*|^{p-1}((1-\bar\zeta_j)\bar\phi_j+(1-\hat\zeta_j)\hat\phi_j)+M(\psi)+V\psi-p\gamma U^{p-1}\psi.\]
For a convolution operator, if the kernel itself is even (resp., odd), then the convolution is also even (resp., odd). Therefore, we need to make a decomposition for the Newton kernel via Taylor expansion. That is, to be more specific,
\[|\mu y+\bar\xi_i-x|^{2-n}=A(x,y)+(n-2)\mu y_1B_1(x)+B_2(x,y),\]
where
\[A(x,y):=\frac{1}{|x-\bar\xi_1|^{n-2}}\left(1-\frac{n-2}{2}\frac{\mu^2|y|^2+2\mu{\sum_{i=2}^ny_i(x-\bar\xi_1)_i}}{|x-\bar\xi_1|^2}\right);\]
\[B_1(x):=-\frac{(x-\bar \xi_1)_1}{|x-\bar\xi_1|^{n}};\  B_2:=O\bigg(\frac{\mu^2|y|^2+2\mu{(y_i,x-\bar\xi_1)}}{|x-\bar\xi_1|^{n+2}}\bigg).\]
Then we define
\[\psi^s(y)=\int A(x,y)W(\psi)(x)dx,\ \psi^*(y)=\int ((n-2)y_1B_1+B_2)(x,y)W(\psi)(x)dy,\]
where $\psi^s$ is even with respect to $y_1$. 

Now we focus on the point-wise estimate of $\psi^*$. Calculating as in Proposition 4.9 in \cite{MEDINA2021}, we have
\[\bigg|\int B_{1}(x)(p\gamma|u_*|^{p-1}((1-\bar\zeta_j)\bar\phi_j+(1-\hat\zeta_j)\hat\phi_j)+V\psi-p\gamma U^{p-1}\psi)\bigg|=||\psi||_{*}+||\bar{\bar\phi}||_*),\]
\[\bigg|\int B_{1}(x)\bigg(1-\sum_{i=1}^k(\bar\zeta_i+\hat\zeta_i)\bigg)N(\phi)\bigg|\leq C(||\bar{\bar\phi}_1||_*+||\psi_1||_*),\]
and 
\[\int B_{2}(x,y)W(\psi)(x)dx=o_k(1)\mu d|y|^2.\]
We only need to show that the remaining terms $\int B_1(1-\sum_{i=1}^k\bar\zeta_i-\sum_{i=1}^k\hat\zeta_i)E=o_k(1)d$.

On the support of $\Lambda=1-\sum_{i=1}^k\bar\zeta_i-\sum_{i=1}^k\hat\zeta_i$, we have shown that 
\begin{align*}
    |E-\tilde E|&\leq CU^{p-2}\Big(\sum_{j=1}^k(\bar U_j+\hat U_j)\Big)^2+C\Big(\sum_{j=1}^k(\bar U_j+\hat U_j)\Big)^p\\
    &\leq CU^{p-2}\sum_{j=1}^k(\bar U_j^2+\hat U_j^2)+C\sum_{j=1}^k(\bar U_j^p+\hat U_j^p)
\end{align*}
where 
\begin{align*}
    \tilde E=-pU^{p-1}\sum_{j=1}^k(\bar U_j+\hat U_j)-\sum_{j=1}^k(\bar U_j^p+\hat U_j^p).
\end{align*}
% We will estimate the main term $\int B_1(1-\sum_{i=1}^k\bar\zeta_i-\sum_{i=1}^k\hat\zeta_i)\tilde E=B_{11}+B_{12}$ in detail, where 
% \[B_{11}:=\int B_{1}(x,y)\bigg(1-\sum_{i=1}^k\bar\zeta_i\bigg)\tilde E,\quad B_{12}:=\sum_{i=1}^k\int B_{1}(x,y)\hat\zeta_i\tilde E,\]
We rewrite $\int B_1(1-\sum_{i=1}^k\bar\zeta_i-\sum_{i=1}^k\hat\zeta_i)\tilde E=\sum_{i=1}^4I_i$ where
\begin{align*}
    &I_1:=\sum_{j=1}^k\int B_1\bigg(1-\bar\zeta_j\bigg)pU^{p-1}\bar U_j,\quad I_2:=-\sum_{j=1}^k\int B_1\Big(\sum_{i=1,i\neq j}^k\bar \zeta_i+\sum_{i=1}^k\hat \zeta_i\Big)pU^{p-1}\bar U_j
\end{align*}
\begin{align*}
    I_3:=\int B_1\Lambda pU^{p-1}\sum_{j=1}^k\hat U_j,\quad I_4=\int B_1 \Lambda\sum_{j=1}^k(\bar U_j^p+\hat U_j^p)
\end{align*}
% \begin{align*}
%     & I_5:=\int B_1\bigg(1-\sum_{i=1}^k(\bar\zeta_i+\hat \zeta_i)\bigg)\sum_{j=1}^k\bar U_j^p,\quad I_6:=\int B_1\bigg(1-\sum_{i=1}^k(\bar\zeta_i+\hat \zeta_i)\bigg)\sum_{j=1}^k\hat U_j^p
% \end{align*}
% \begin{align*}
%     &I_1:=\sum_{j=1}^k\int B_1\bigg(1-\bar\zeta_j\bigg)pU^{p-1}\bar U_j,\quad I_2:=\int B_1\bigg(1-\sum_{i=1}^k\hat\zeta_i\bigg)pU^{p-1}\sum_{j=1}^k\hat U_j\\
%     &I_3:=-\sum_{j=1}^k\int B_1\Big(\sum_{i=1,i\neq j}^k\bar \zeta_i+\sum_{i=1}^k\hat \zeta_i\Big)pU^{p-1}\bar U_j,\quad I_4:=-\int B_1\sum_{i=1}^k\bar \zeta_ipU^{p-1}\sum_{j=1}^k\hat U_j\\
%     & I_5:=\int B_1\bigg(1-\sum_{i=1}^k(\bar\zeta_i+\hat \zeta_i)\bigg)\sum_{j=1}^k\bar U_j^p,\quad I_6:=\int B_1\bigg(1-\sum_{i=1}^k(\bar\zeta_i+\hat \zeta_i)\bigg)\sum_{j=1}^k\hat U_j^p
% \end{align*}
%The first term is given by 
%\[\bigg|\int B_{1}\bigg(1-\sum_{i=1}^k\bar\zeta_i\bigg)U^{p-1}\bar U_j\bigg|\leq\bigg|\int B_1(1-\bar\zeta_j)U^{p-1}\bar U_j\bigg|+\sum_{l=1,j\neq j}^k\bigg|\int B_{1}\bar\zeta_lU^{p-1}\bar U_j\bigg|.\]
Let us deal with $I_1$ first.
When $j=1$, standard estimates yield
\begin{align*}
    0\leq \int B_1(1-\bar \zeta_1)U^{p-1}\bar U_1\leq \mu^{\frac{n-2}{2}}\int_{|x-\bar\xi_1|>\delta/k}\frac{1}{|x-\bar\xi_1|^{2n-3}} dx=o_k(1)d.
\end{align*}
For $j>1$, by the decay property of $U$ and $\bar U_j$, it is easy to see that the following holds for any $\rho>0$,
\[\int_{\partial B(\bar \xi_j,\rho)}B_1U^{p-1}\bar U_j\geq0,\quad \int_{\partial B(\bar \xi_j,\rho)}B_1 U^{p-1}>0.\]
Thus the Fubini theorem implies that
\begin{align*}
0\leq \int B_1(1-\zeta_j)U^{p-1}\bar U_j\leq \int B_1U^{p-1}\bar U_j\leq \int B_1U^{p-1}\frac{\mu^{\frac{n-2}2}}{|x-\bar\xi_j|^{n-2}},
\end{align*}
And again, by Fubini theorem, and the fact that\(\int_{\partial B(0,\rho)}B_1{|x-\bar\xi_j|^{2-n}}>0\) , we have
\[\int B_1U^{p-1}\frac{\mu^{\frac{n-2}2}}{|x-\bar\xi_j|^{n-2}}\leq C\int B_1\frac{\mu^{\frac{n-2}2}}{|x-\bar\xi_j|^{n-2}}\leq C\mu^{\frac{n-2}{2}}\int\frac{(x-\bar\xi_1)_1}{|x-\bar \xi_1|^{n}|x-\bar \xi_j|^{n-2}} \]
Applying the following fact that
\begin{align*}
    \int_{\mathbb{R}^n}\frac{(x-a)_1}{|x-a|^p|x-b|^q}dx=C(p,q,n)\frac{(a-b)_1}{|a-b|^{p+q-n}},\quad \text{if }p+q>n+1\ \ \ \ \ \ \ \ \ \ \ \\
   \left|\int_{B(a, \rho)}\frac{(x-a)_1}{|x-a|^p|x-b|^q}dx\right|\leq C(p,q,n,\rho)(|\log|a-b||+1),\quad \text{if }p+q=n+1,\ 
\end{align*}
for any $\rho\gg|a-b|$, $0<p\leq n$, $0\leq q<n$ and $a\neq b$, then
\begin{align*}
    \bigg|\int B_1(1-\bar\zeta_j)U^{p-1}\bar U_j\bigg|\leq\begin{cases}
        C\mu^{\frac{n-2}2}|\bar\xi_1-\bar\xi_j|^{2-n}|(\bar\xi_1-\bar\xi_j)_1|&\text{if }n\geq 4\\
        C\mu^{\frac{1}2}|\log|\bar\xi_1-\bar\xi_j||&\text{if }n=3.
    \end{cases}
\end{align*}

Using the summation formula, we have 
\begin{equation*}
    \bigg|\sum_{j=2}^k\int B_1(1-\bar\zeta_j)U^{p-1}\bar U_j\bigg|\leq
    \begin{cases}
        C\log^{-1}k,&n=3\\
        Ck^{-1},&n=4\\
        Ck^{-2}\log \log k,& n=5\\
        Ck^{-2},& n\geq6
    \end{cases}
    =o_k(1) d.
\end{equation*}
Thus we have proved $|I_1|=o_k(1) d$. 

Consider $I_2$. First when $j=1$, one has 
\begin{align*}
    \left|\int B_1\Big(\sum_{i=2}^k\bar \zeta_i+\sum_{i=1}^k\hat \zeta_i\Big) U^{p-1}\bar U_1\right|\leq \mu^{\frac{n-2}{2}}\int_{|x-\bar\xi_1|>\delta/k}\frac{1}{|x-\bar\xi_1|^{2n-3}} dx=o_k(1)d.
\end{align*}
Second, when $j>1$
% \begin{align*}
%    0\leq  \int B_1\sum_{i=1,i\neq j}^k\bar \zeta_i U^{p-1}\bar U_j\leq \int B_1U^{p-1}\bar U_j\leq C\mu^{\frac{n-2}{2}}\int \frac{1}{|x-\bar\xi_1|^{n-1}|x-\hat\xi_j|^{n-2}}
% \end{align*}
\begin{align*}
\left|\int B_1\bar \zeta_i U^{p-1}\bar U_j\right|&\leq \begin{cases}C\mu^{\frac{n-2}{2}}|\bar \xi_i-\bar \xi_1|^{1-n}|\bar \xi_i-\bar \xi_j|^{2-n}k^{-n},&\text{if } i\neq 1\\
C\mu^{\frac{n-2}{2}}|\bar \xi_1-\xi_j|^{2-n}k&\text{if }i=1
\end{cases}\\
    \left|\int B_1\hat \zeta_i U^{p-1}\bar U_j\right|&\leq C\mu^{\frac{n-2}{2}}|\hat \xi_i-\bar \xi_1|^{1-n}|\hat \xi_i-\bar \xi_j|^{2-n}|B(\hat \xi_i,\delta/k)|
\end{align*}
Summing with respect to $j$, then $i$, yields
\begin{align*}
    |I_2|\leq o_k(1) d +\begin{cases}
        C\mu^{\frac{n-2}{2}}k^{n-3},&\text{if }n\geq 4\\
        C\mu^{\frac{1}{2}}\log k,&\text{if } n=3
    \end{cases}=o_k(1) d 
\end{align*}
Let us consider $I_3$. For $I_3$, we have
\begin{align*}
    |I_3|\leq C\sum_{j=1}^k\int \frac{1}{|x-\bar\xi_1|^{n-1}}\frac{\lambda^{\frac{n-2}{2}}}{|x-\hat \xi_1|^{n-2}}.
\end{align*}
% \[\bigg|\int B_{1}(x,y)\bigg(1-\sum_{i=1}^k\bar\zeta_i\bigg)U^{p-1}\hat U_j\bigg|\leq\int |B_1|U^{p-1}\hat U_j\leq C\mu^{\frac n2}|y|\int \frac{1}{|x-\bar\xi_1|^{n-1}|x-\hat\xi_j|^{n-2}}\]
% For $n\geq4$, standard calculation gives
% \begin{equation*}
% \bigg|\int B_1U^{p-1}\hat U_j\bigg|\leq C\mu^{\frac n2}|y|\int \frac{U^{p-1}}{|x-\bar\xi_1|^{n-1}|x-\hat\xi_j|^{n-2}}\leq\left\{
% \begin{aligned}
%     C\mu^{\frac32}\log|\bar\xi_1-\hat\xi_j||y|,\ n=3;\\
%     C\mu^{\frac n2}|\bar\xi_1-\hat\xi_j|^{3-n}|y|, \ \ n\geq4,
% \end{aligned}\right.
% \end{equation*}
Applying the following fact that
\begin{align*}
    \int_{\mathbb{R}^n}\frac{1}{|x-a|^p|x-b|^q}dx=\frac{C(p,q,n)}{|a-b|^{p+q-n}} ,\quad \text{if }p+q>n,\ \ \ \ \ \ \ \ \ \ \ \ \ \\
    \left|\int_{B(a,\rho)}\frac{1}{|x-a|^p|x-b|^q}dx\right|\leq C(p,q,\rho,n)(|\log |a-b||+1),\quad \text{if }p+q=n,
\end{align*}
where $\rho\gg|a-b|$, $0<p\leq n$, $0\leq q<n$ and $a\neq b$, 
%for any $0\leq p,q<n$ and $a\neq b$
%where the second inequality is due to a standard dilation argument shown below. 
% \textcolor{blue}{Consider the general model, say
% \[J(a,b)=\int\frac{1}{|x-a|^p|x-b|^q}dx,\ p+q> n,\ 0<p,q< n,\ a,b\in B_2.\]
% A substitution $x=|a-b| z+a$ gives
% \[J(a,b)=|a-b|^{n-p-q}\int\frac{1}{|z|^p|z+e(a,b)|^q}dz;\ e(a,b)=\frac{a-b}{|a-b|}.\]
% A direct application of the convolution theorem of Fourier transform, we have 
% \[\int\frac{1}{|z|^p|z+e(a,b)|^q}dz=C(p,q,n)\frac{1}{|e(a,b)|^{p+q-n}}=C(p,q,n).\]
%  Consider the truncated model, say
% \[J_T(a,b)=\int_{B_2}\frac{1}{|x-a|^p|x-b|^q}dx,\ p+q= n,\ 0< p,q< n,\ a,b\in B_2.\]
% A substitution $x=|a-b| z+a$ gives
% \[J_T(a,b)=\int_{B_{\frac{2}{|a-b|}}(a)}\frac{1}{|z|^p|z+e(a,b)|^q}dz\leq C(n,p,q)|\log|a-b||,\]
% and we also have
% \[\bigg|\int_{\mathbb R^n\setminus B_2}\frac{U^{p-1}}{|x-a|^p|x-b|^q}dx\bigg|\leq C.\]}
%For $n=3$, using the Riemann sum, one has
%\[\bigg|\sum_{j=2}^k \log|\bar\xi_1-\hat\xi_j|\bigg|\leq Ck\int_{0}^1|\log (d^2+4\sin^2 \pi s)|ds\leq Ck,\]
%also for $n\geq 4$,
% we obtain
% \begin{align*}
%     \bigg|\int B_{1}(x,y)\bigg(1-\sum_{i=1}^k\bar\zeta_i\bigg)U^{p-1}\hat U_j\bigg|\leq \begin{cases}
%         C\mu^{\frac n2}|\bar\xi_1-\hat\xi_j|^{3-n}|y|&\text{if }n\geq 4,\\
%         C\mu^{\frac32}\log|\bar\xi_1-\hat\xi_j||y|&\text{if }n=3.
%     \end{cases}
% \end{align*}
we have
\begin{align*}
    |I_3|\leq C\sum_{i=1}^k\mu^{\frac{n-2}{2}}|\bar\xi_1-\hat\xi_j|^{3-n}\text{ if  }n\geq 4,\quad |I_3|\leq C\sum_{i=1}^k\mu^{\frac{n-2}{2}}\log |\bar\xi_1-\hat\xi_j|\text{ if  }n=3
\end{align*}
Then a simple summation using \eqref{sumB} gives 
\begin{equation*}
    |I_3|\leq
    \begin{cases}
    C\log^{-1}k,&\text{if } n=3\\
    Ck^{-1}\log\log k,&\text{if }n=4\\
    Ck^{3-n}d^{4-n},&\text{if } n\geq5
    \end{cases}
    =o_k(1)d.
\end{equation*}

For $I_4$, one can use the estimate the pointwise relationships on the support of $\Lambda$,
\[\hat U_j^p\leq\frac{\lambda^{\frac{n+2}{2}}}{|x-\hat\xi_j|^{n-1}},\ \bar U_j^p\leq\frac{\mu^{\frac{n+2}{2}}}{|x-\bar\xi_j|^{n-2}}.\]
This is reduced to the previous analysis since 
\begin{align*}
    |I_4|\leq C\int_{|x-\bar\xi_1|>\delta/k}\frac{\mu^{\frac{n+2}{2}}}{|x-\bar \xi_1|^{2n-3}}+C\int \frac{1}{|x-\bar \xi_1|^{n-1}}\Big(\frac{\mu^{\frac{n+2}{2}}}{|x-\bar \xi_j|^{n-2}}+\frac{\lambda^{\frac{n+2}{2}}}{|x-\hat \xi_j|^{n-2}}\Big)
\end{align*}
Finally, $\int B_1\Lambda |E-\tilde E|$ can be estimated via similar approach for $I_1,I_2,I_3,I_4$.
\end{proof}
%%%%%%%%%%%%%%%%%%%%%%%%%%%%%%%%%%%%%%%%%%%%%%%%%%%%%

\section{Balancing the parameters}
We need to adjust the value of the coefficients to meet the condition of solvability. In fact, this problem is thus reduced to solving a finite-dimensional problem, namely
\begin{equation}
\int (\bar\zeta_1 E+\gamma\bar{\mathcal N}(\bar\phi_1,\hat\phi_1,\phi))\bar Z_1=0;\quad\int (\bar\zeta_1 E+\gamma\bar{\mathcal N}(\bar\phi_1,\hat\phi_1,\phi))\bar Z_{n+1}=0.
\end{equation}

To solve the system, deriving the asymptotic formulae for both integrals is necessary, which spell as
\begin{equation}
    \int (\bar\zeta_{1} E+\gamma\bar{\mathcal N}(\bar\phi_1,\hat\phi_1,\phi))\bar Z_{n+1}=\left\{
\begin{aligned}
k^{2-n}l((-C_{1,n}l+C_{3,n})+\Theta_{k}(l,t)),\ \ \ \ \ \ \ \ \ &n\geq4;\\
(k\log k)^{-1}l((-C_{1,3}l+C_{3,3})+\Theta_{k}(l,t)),\ &n=3,
\end{aligned}
\right.
\end{equation}
\begin{equation}
 \int (\bar\zeta_1 E+\gamma\bar{\mathcal N}(\bar\phi_1,\hat\phi_1,\phi))\bar Z_1=\left\{
\begin{aligned}
   k^{-n-\frac{n-3}{n-1}}tl^{\frac n{n-2}}((C_{4,n}-C_{5,n}t^{1-n}\hat l)+\Theta(l,t)),\ n\geq4;\\
   k^{-3}\log^{-\frac72}tl^3((C_{4,3}-C_{5,3}t^{-2}\hat l)+\Theta(l,t)),\ \ n=3;
\end{aligned}
\right.
\end{equation}
where $\Theta_k$ are generic functions that uniformly converge to 0 as $k\to\infty$. Invoking the fixed point theorem of subsets of $\mathbb{R}^2$, also the remark 11 at the end of this section, the main theorem follows. 

Now we turn to the proof of the aforementioned asymptotic formulae. In the proofs given below, an more precise version of summing formula, given by \cite{MEDINA2021} in Appendix B, is frequently used, that is
\begin{equation}
\sum_{j=2}^k\bigg|\frac{\bar\xi_1-\bar\xi_j}{r}\bigg|^{2-n}=\left\{
\begin{aligned}
A_nk^{n-2}(1+O(\sigma_k)),\ \ \ n\geq4;\\
A_3k\log k(1+O(\sigma_k)),\ n=3,
\end{aligned}\right.
\ \ \sigma_k=\left\{
\begin{aligned}
    k^{-2},\ \ \ \ \ \ \ \ \ n\geq6;\\
    k^{-2}\log k,\ n=5;\\
    k^{-1},\ \ \ \ \ \ \ \ \ n=4;\\
    \log^{-1}k,\ \ \ \ n=3,
\end{aligned}\right.
\end{equation}
where $A_n$ is a constant that only depends on $n$. Moreover, a similar summing formula is also introduced here,
\begin{equation}\label{sumB}
\sum_{j=1}^k|\bar\xi_1-\hat\xi_j|^{2-n}=\left\{
\begin{aligned}
B_{n,r,R}kd^{3-n}(1+O(\omega_k)),\ \ \ \ \ \ \ \ \ \ \ \ \ \ \ n\geq4;\\
B_{3,r,R}k\log \Big(\frac\pi d\Big)(1+O(\log^{-1}d)),\ n=3,
\end{aligned}
\right.
\ \ \ \omega_k=\left\{
\begin{aligned}
(dk)^{-2},\ n\geq5;\\
d,\ \ \ \ \ \ \ \ \ \ n=4.
\end{aligned}
\right.
\end{equation}
where $B_{n,r,R}$ is a constant that only depends on $n,r,R$.
We start from the easier one.\\
\textit{Proof of Formula} (44). For (44), we are ready to directly derive the formula.
Now assume $n\geq4$. The numerator is
\[\int (\bar\zeta_1 E+\gamma\bar{\mathcal N}(\bar\phi_1,\hat\phi_1,\phi))\bar Z_{n+1}=\int E\bar Z_{n+1}-\int(1-\bar\zeta_1)E\bar Z_{n+1}+\gamma\int\bar{\mathcal N}(\bar\phi_1,\hat\phi_1,\phi)\bar Z_{n+1}.\]
The first term, as was treated in Section 3, can be divided as follows:
\[\int E\bar Z_{n+1}=\left(\int_{\bar B_1}+\sum_{j\neq1}\int_{\bar B_j}+\sum_{j=1}^k\int_{\hat B_j}+\int_{\text{Ext}}\right)E\bar Z_{n+1}.\]
We will show that the first term is the main term. After changing the variables as in Section 3, we arrive at a easier form for computation. By denoting 
\[V(y)=-\sum_{j\neq1}\bar{\bar U}_j-\sum_{j=1}^{k}\bar{\hat U}_j+\bar{\bar U}\]
where
\[\bar{\bar U}_j(y)=\mu^{\frac{n-2}2}\bar U_j(\mu y+\bar\xi_1);\ \bar{\hat U}_j(y)=\mu^{\frac{n-2}2}\hat U_j(\mu y+\bar\xi_1);\ \bar{\bar U}=\mu^{\frac{n-2}2}U(\mu y+\bar\xi_1),\]
we have 
\begin{equation}
\begin{aligned}
&\quad\gamma^{-1}\int_{\bar B_1}E\bar Z_{n+1}=\gamma^{-1}\mu^{\frac{n-2}{2}}\int_{B(0,\frac{\delta}{k\mu})}E(\mu y+\bar\xi_1) Z_{n+1}\\
&=p\sum_{j\neq1}\int_{B(0,\frac{\delta}{k\mu})}U^{p-1}\bar{\bar U}_jZ_{n+1}+p\sum_{j=1}^{k}\int_{B(0,\frac{\delta}{k\mu})}U^{p-1}\bar{\hat U}_jZ_{n+1}\\
&\quad -p\int_{B(0,\frac{\delta}{k\mu})}U^{p-1}\bar{\bar U}Z_{n+1}+p\int_{B(0,\frac{\delta}{k\mu})}((U+sV)^{p-1}-U^{p-1})VZ_{n+1}\\
&\quad +\sum_{j\neq1}\int_{B(0,\frac{\delta}{k\mu})}\bar{\bar U}_j^pZ_{n+1}+\sum_{j=1}^{k}\int_{B(0,\frac{\delta}{k\mu})}\bar{\hat U}_j^pZ_{n+1}+\int_{B(0,\frac{\delta}{k\mu})}\bar{\bar U}^pZ_{n+1}.
\end{aligned}
\end{equation}
We will show that the main terms are the first and the third terms. To simplify the notations, we denote $I_1:=-\int U^{p-1}Z_{n+1}$, then a simple calculation shows that $$\int_{B(0,\frac\delta{\mu k})} U^{p-1}Z_{n+1}=-I_1+O(k^{-2}),$$
and we have
\begin{equation}
\begin{aligned}
\sum_{j=2}^k\int_{B(0,\frac{\delta}{k\mu})}U^{p-1}\bar{\bar U}_jZ_{n+1}=-2^{\frac{n-2}2}\mu^{n-2}I_1\bigg(\sum_{j=2}^k\frac{1+\mu^2|\bar\xi_1-\bar\xi_j|^{-2}O(\log k)}{|\bar\xi_1-\bar\xi_j|^{n-2}}\bigg);\\
\sum_{j=1}^k\int_{B(0,\frac{\delta}{k\mu})}U^{p-1}\bar{\hat U}_jZ_{n+1}=-C_2(\mu\lambda)^{\frac{n-2}2}\bigg(\sum_{j=1}^k\frac{1+\mu^2|\bar\xi_1-\hat\xi_j|^{-2}O(\log k)}{|\bar\xi_1-\hat\xi_j|^{n-2}}\bigg);\ \\
\int_{B(0,\frac{\delta}{k\mu})}U^{p-1}\bar{\bar U}Z_{n+1}=-\mu^\frac{n-2}2U(\bar\xi_1)(I_1+O(k^{-2}))\bigg(1+\mu^2O(\log k)\bigg).\ \ \ \ \ \ \ 
\end{aligned}
\end{equation}
These three formulae are direct results from the asymptotic formulae and the summation formulae. Nota bene, the second term reads
\begin{equation*}
\begin{aligned}
&\bigg|\sum_{j=1}^k\int_{B(0,\frac{\delta}{k\mu})}U^{p-1}\bar{\hat U}_jZ_{n+1}\bigg|=C_2(\mu\lambda)^{\frac{n-2}2}\bigg(\sum_{j=1}^k\frac{1+\mu^2|\bar\xi_1-\hat\xi_j|^{-2}O(\log k)}{|\bar\xi_1-\hat\xi_j|^{n-2}}\bigg)\\
=&\ Ck^{5-2n}d^{3-n}+o(k^{5-2n}d^{3-n})=o(k^{2-n}).
\end{aligned}
\end{equation*}
Therefore, only the first and the third terms are the main terms. We now derive finer expression of the coefficients for both expansions. For the first term
\begin{align*}
&\ \ \sum_{j\neq1}\int_{B(0,\frac{\delta}{k\mu})}U^{p-1}\bar{\bar U}_jZ_{n+1}=-2^{\frac{n-2}2}\mu^{n-2}I_1\sum_{j\neq1}\frac{1+\mu^{2}|\bar\xi_1-\bar\xi_j|^{-2}O(k\log k)}{|\bar\xi_1-\bar\xi_j|^{n-2}}\\
=&-2^{\frac{n-2}2}A_nI_1l^2k^{2-n}r^{n-2}(1+O(k^{-2}\log k))(1+O(\sigma_k))\\
=&-(1+O(k^{-2}\log k)+O(\sigma_k))
\end{align*} 
The third is 
\begin{align*}
&\int_{B(0,\frac{\delta}{k\mu})}U^{p-1}\bar{\bar U}Z_{n+1}=-\mu^\frac{n-2}2U(\bar\xi_1)I_1(1+\mu^2(1+|\bar\xi_1|^2)^{-1}\Theta_{k}(l,t))\\
=&-U(\bar\xi_1)lI_1k^{2-n}(1+O(k^{-2}))(1+O(k^{-4}\log k))=-U(\bar\xi_1)lI_1k^{2-n}(1+O(k^{-2})).
\end{align*}
We denote
\[C_{1,n}=2^{\frac{n-2}2}A_nI_1l^2,\ C_{3,n}=U(\textbf{e}_1)lI_1.\]
A simple application of the summation we used before, as well as the main term of the asymptotic formulae, we have 
\small{
\begin{equation*}
\sum_{j\neq1}\int_{B(0,\frac{\delta}{k\mu})}\bar{\bar U}_j^pZ_{n+1}=O(k^{-n});\ 
\sum_{j=1}^{k}\int_{B(0,\frac{\delta}{k\mu})}\bar{\hat U}_j^pZ_{n+1}=O(k^{-n});\ 
\int_{B(0,\frac{\delta}{k\mu})}\bar{\bar U}^pZ_{n+1}=O(k^{-n}).
\end{equation*}}
Applying the three formulae above yields
\begin{equation}
p\bigg|\int_{B(0,\frac{\delta}{k\mu})}((U+sV)^{p-1}-U^{p-1})VZ_{n+1}\bigg|=O(k^{-n}).
\end{equation}
For the last term, we process as in \cite{MEDINA2021}. Using the norm estimates in Section 4, and the fact that
\[|Z_{n+1}|\leq CU,\]
also H\"older inequality, we arrive at 
\[\left|\int_{\text{Ext}}E\bar Z_{n+1}\right|=O(k^{1-n});\ 
\left|\sum_{j\neq1}\int_{\bar B_i}E\bar Z_{n+1}\right|=O(k^{1-n});\ 
\left|\sum_{j=1}^k\int_{\hat B_j}E\bar Z_{n+1}\right|=O(k^{1-n}).\]
Combining the previous calculation, we arrive at
\begin{equation}
    \int E\bar Z_{n+1}=k^{2-n}l((-C_{1,n}l+C_{3,n})+\Theta_{k}(l,t)),
\end{equation}
which is, as we will prove below, the main term of the second integration.
For the second term, we have
\[\left|\int(\bar\zeta_1-1)E\bar Z_{n+1}\right|\leq C\left|\int_{\bar B_1^C} E\bar Z_{n+1}\right|.\]
Dividing the domain into 3 different parts as in Section 3, and plugging in the calculation of the last part of the first term, we arrive at 
\[\left|\int(\bar\zeta_1-1)E\bar Z_{n+1}\right|=O(k^{1-n})\Theta_k(l,t).\]
For the third term, using the decomposition in the proof for Proposition 6, we arrive at
\[\left|\int \bar{\mathcal N}\bar Z_{n+1}\right|\leq Ck^{3-n-\frac nq}=o(k^{1-n})\Theta_{k}(l,t).\]

Now that we have concluded the proof for $n\geq4$, we will deal with the cases in which $n=3$. The idea is to use the estimates for 3D instead of the one we used before, and the process repeats the one we performed before. The only difference lies in the last term, where we only need to apply H\"older inequality after changing variables to obtain the result, that is 
\[\left|\int \bar{\mathcal N}\bar Z_{n+1}\right|\leq C||\bar{\bar {\mathcal N}}||_{**}\bigg(\int \frac{1}{1+|y|^{2n}}\bigg)^{1-\frac1q}\leq Ck^{-2}\log^{-2}k.\]

So far, we may be able to solve $l$ from the expansion. In fact, a simple calculation gives the solution 
\[l=2^{\frac{2-n}2}U(\bar\xi_1)A_n^{-1}r^{{2-n}}(1+o(1)).\]
\noindent\textit{Proof of Formula} (45).\par
For the formula, the proof highly relies on the symmetry decomposition.
The numerator is
\[\int (\bar\zeta_1 E+\gamma\bar{\mathcal N}(\bar\phi_1,\phi))\bar Z_1=\int E\bar Z_1-\int(1-\bar\zeta_1)E\bar Z_1+\gamma\int\bar{\mathcal N}(\bar\phi_1,\phi)\bar Z_1.\]
The first term, as was treated in section 3, can be divided as follows:
\[\int E\bar Z_{1}=\left(\int_{\bar B_1}+\sum_{j\neq1}\int_{\bar B_j}+\sum_{j=1}^k\int_{\hat B_j}+\int_{\text{Ext}}\right)E\bar Z_{1}.\]
Proceeding as in the previous part, we have 
\begin{equation}
\left|\int_{\text{Ext}}E\bar Z_{1}\right|=o(k^{-1-n});\ 
\left|\sum_{j\neq1}\int_{\bar B_i}E\bar Z_{1}\right|=O(k^{-1-n});\ 
\left|\sum_{j=1}^k\int_{\hat B_j}E\bar Z_{1}\right|=o(k^{-1-n}).
\end{equation}
Now we turn to the main term. By variable changing, we have
\begin{equation}
\begin{aligned}
\quad\gamma^{-1}\int_{\bar B_1}E\bar Z_{1}=&\gamma^{-1}\lambda^{\frac{n-2}{2}}\int_{B(0,\frac{\delta}{k\mu})}E(\mu y+\bar\xi_1) Z_{1}\\
=&\ \ p\sum_{j\neq1}\int_{B(0,\frac{\delta}{k\mu})}U^{p-1}\bar{\bar U}_jZ_{1}+p\sum_{j=1}^{k}\int_{B(0,\frac{\delta}{k\mu})}U^{p-1}\bar{\hat U}_jZ_{1}\\
&-p\int_{B(0,\frac{\delta}{k\mu})}U^{p-1}\bar{\bar U}Z_{1}+p\int_{B(0,\frac{\delta}{k\mu})}((U+sV)^{p-1}-U^{p-1})VZ_{1}\\
&+\sum_{j\neq1}\int_{B(0,\frac{\delta}{k\mu})}\bar{\bar U}_j^pZ_{1}+\sum_{j=1}^{k}\int_{B(0,\frac{\delta}{k\mu})}\bar{\hat U}_j^pZ_{1}+\int_{B(0,\frac{\delta}{k\mu})}\bar{\bar U}^pZ_{1}.
\end{aligned}
\end{equation}
In fact, unlike the previous calculation, all the first three terms here are the main terms, as we will check next.
We first notice that for a $L^2$ function $g$ and a vector $v$, we have
\[\int_{B_a}(y,v)Z_1(y)g(y)=\int_{B_a}v_1y_1Z_1(y)g(y).\]
Also, we have
\[(\bar\xi_1-\bar\xi_n)_1=r-r\cos{\frac{2\pi(j-1)}{k}}=2r\sin^2\frac{\pi(j-1)}{k},\]
\[\left|\bar\xi_1-\bar\xi_n\right|^2=\left|r-r\exp{\frac{2i\pi(j-1)}{k}}\right|^2=4r^2\sin^2\frac{\pi(j-1)}{k}.\]
From now on, we denote $I_2:=-\int U^{p-1}Z_1y_1$, and we have 
\[-\int_{B(0,\frac{\delta}{k\mu})}U^{p-1}Z_1y_1=I_2+O(k^{-2})\]
We process the first and the third term first, for some reasons that will be clear below. For the first term, we have 
\begin{equation*}
\begin{aligned}
&\sum_{j\neq1}\int_{B(0,\frac{\delta}{k\mu})}U^{p-1}\bar{\bar U}_jZ_{1}\\
=\ &\sum_{j\neq1}\int_{B(0,\frac{\delta}{k\mu})}U^{p-1}Z_1\frac{2^\frac{n-2}{2}\mu^\frac{n-2}{2}\mu^\frac{n-2}{2}}{|\bar\xi_1-\bar\xi_j|^{n-2}}\left(-\frac{(n-2)(y,\bar\xi_1-\bar\xi_j)\mu}{|\bar\xi_1-\bar\xi_j|^2}+O\left(\frac{\mu^2(1+|y|^2)}{|\bar\xi_1-\bar\xi_j|^2}\right)\right)\\
=\ &\sum_{j\neq1}\int_{B(0,\frac{\delta}{k\mu})}-(n-2)U^{p-1}Z_1y_1\frac{2^\frac{n-2}{2}\mu^{n-1}}{|\bar\xi_1-\bar\xi_j|^{n-2}}\left(\frac1{2r}+O\left(\frac{\mu^2(1+|y|^2)}{|\bar\xi_1-\bar\xi_j|^2}\right)\right)\\
=\ &(n-2)2^{\frac{n-4}2}r^{n-3}A_nI_2l^{\frac{n}{n-2}}lk^{-n}(1+O(k^{-2}\log k)+O(\sigma_k)).
\end{aligned}
\end{equation*}
The last term directly reads
\begin{align*}
\int_{B(0,\frac{\delta}{k\mu})}U^{p-1}\bar{\bar U}Z_{1}&=\frac{(n-2)\mu^{\frac n2} r}{1+r^2}U(\bar\xi_1)\int_{B(0,\frac{\delta}{k\mu})} -U^{p-1}y_1Z_{n+1}(1+O(k^{-4}\log k))\\
&=\frac{(n-2)\mu^{\frac n2} rI_2}{1+r^2}U(\bar\xi_1)(1+O(k^{-2})+O(k^{-4}\log k))\\
&=\frac{(n-2)l^{\frac n{n-2}} rI_2}{1+r^2}U(\bar\xi_1)k^{-n}(1+O(k^{-2})).
\end{align*}
Recall that the value of $l$ has been calculated before, then it is easy to verify that the difference between the coefficients of $l^{\frac n{n-2}}k^{-n}$ is
$$\frac{(n-2)2^{\frac{n-2}2}r^{n-2}A_nI_2l}{2r}-\frac{(n-2) rI_2}{1+r^2}U(\bar\xi_1)=\frac{n-2}{2r}U(\bar\xi_1)I_2\frac{1-r^2}{1+r^2}(1+o(1)).$$
By invoking Remark 3, we see that $(1-r^2)/(1+r^2)\sim d/2$. That is, we have 
\[\sum_{j\neq1}\int_{B(0,\frac{\delta}{k\mu})}U^{p-1}\bar{\bar U}_jZ_{1}-\int_{B(0,\frac{\delta}{k\mu})}U^{p-1}\bar{\bar U}Z_{1}=C_{4,n}l^{\frac{n}{n-2}}tk^{-n-\frac{n-3}{n-1}}(1+o(1)),\]
where \(C_{4,n}=\frac{n-2}{4}U(\textbf{e}_1)I_2\)
For the second term, the cosine law gives
\[(\bar\xi_1-\hat\xi_n)_1=r-R\cos{\frac{2\pi(j-1)}{k}}=2R\sin^2\frac{\pi(j-1)}{k}-d,\]
\[\left|\bar\xi_1-\hat\xi_n\right|=\left|r-R\exp{\frac{2i\pi(j-1)}{k}}\right|^2=d^2+4Rr\sin^2 \frac{\pi(j-1)}{k}.\]
Plugging these quantities into the expansion, we have
\begin{equation*}
\begin{aligned}
&\sum_{j=1}^k\int_{B(0,\frac{\delta}{k\mu})}U^{p-1}\bar{\hat U}_jZ_{1}\\
=\ &\sum_{j=1}^k\int_{B(0,\frac{\delta}{k\mu})}U^{p-1}Z_1\frac{2^\frac{n-2}{2}\lambda^\frac{n-2}{2}\mu^\frac{n-2}{2}}{|\bar\xi_1-\hat\xi_j|^{n-2}}\left(-\frac{(n-2)(y,\bar\xi_1-\hat\xi_j)\mu}{|\bar\xi_1-\hat\xi_j|^2}+O\left(\frac{\mu^2(1+|y|^2)}{|\bar\xi_1-\hat\xi_j|^2}\right)\right)\\
=\ &-\sum_{j=1}^k\int_{B(0,\frac{\delta}{k\mu})}(n-2)U^{p-1}Z_1y_1\frac{2^\frac{n-2}{2}\lambda^\frac{n-2}{2}\mu^\frac{n}{2}}{(d^2+4Rr\sin^2 \frac{\pi(j-1)}{k})^{\frac{n-2}2}}\times\\
&\qquad\qquad\qquad\qquad\qquad\qquad\qquad\times\left(\frac{2R\sin^2\frac{\pi(j-1)}{k}-d}{d^2+4Rr\sin^2 \frac{\pi(j-1)}{k}}+O\left(\frac{\mu^2(1+|y|^2)}{|\bar\xi_1-\hat\xi_j|^2}\right)\right)\\
=\ &-C_{5,n}\mu^{\frac n2}\lambda^{\frac{n-2}2}kd^{2-n}+o(k^{-n-\frac{n-3}{n-1}})=-C_{5,n}t^{2-n}l^{\frac{n}{n-2}}\hat lk^{-n-\frac{n-3}{n-1}}+o(k^{-n-\frac{n-3}{n-1}}).
\end{aligned}
\end{equation*}
The last equality is obtained via the formula \eqref{sumB}. Extra caution is required to determine the sign of this term, as the dominant term in the brackets tends to $-d(d^2+4Rr\sin^2((j-1)\pi/k))^{-1}$ as $d\to0$, which is negative. 

Similar to the calculation above, we may write down the higher order terms,
\small{
\begin{equation*}
\sum_{j\neq1}\int_{B(0,\frac{\delta}{k\mu})}\bar{\bar U}_j^pZ_{1}=O(k^{-n-1});\ 
\sum_{j=1}^{k}\int_{B(0,\frac{\delta}{k\mu})}\bar{\hat U}_j^pZ_{1}=O(k^{-n-1});\ 
\int_{B(0,\frac{\delta}{k\mu})}\bar{\bar U}^pZ_{1}=O(k^{-n-1}).
\end{equation*}
}
Applying the three formulae above yields
\begin{equation*}
p\bigg|\int_{B(0,\frac{\delta}{k\mu})}((U+sV)^{p-1}-U^{p-1})VZ_{1}\bigg|=O(k^{-n-1}).
\end{equation*}

Combining the previous proof, and repeating the process for the case $n=3$, we arrive at
\begin{equation}
 \int E\bar Z_1=\left\{
\begin{aligned}
   &k^{-n-\frac{n-3}{n-1}}tl^{\frac n{n-2}}((C_{4;6,n}-C_{5,n}t^{1-n}\hat l)+\Theta(l,t)),\ n\geq4;\\
   &k^{-3}\log^{-\frac72}tl^3((C_{4;6,3}-C_{5,3}t^{-2}\hat l)+\Theta(l,t)),\ \ n=3.
\end{aligned}
\right.
\end{equation}

For the second term, we have
\[\left|\int(\bar\zeta_1-1)E\bar Z_{n+1}\right|\leq C\left|\int_{\bar B_1^C} E\bar Z_{n+1}\right|.\]
By separating the domain, and plugging in (52), we have 
\[\left|\int_{\bar B_1^C} E\bar Z_{n+1}\right|\leq O(k^{-1-n}),\]
which gives the desired estimate.

We now proceed to handle the final term by applying the results from Section 5. Following the approach in Section 4, we divide $\bar{\mathcal N}$ into four parts
\[\bar{\mathcal N}(\bar\phi_1,\phi)=f_1+f_2+f_3+f_4.\]
We assume that $n\geq 5$ first. 
Assuming first that $n\geq 5$, we estimate $f_1$. By a change of variable and a pointwise estimate of $f_1$, we obtain
\[\left|\int f_1(y)\bar Z_1(y)\right|\leq C\int_{B_{\delta (k\mu)^{-1}}}U^{p-2}|Z_1||\bar{\bar\phi}_1|.\]
Also, in this region, we have $\sqrt \mu\leq C|y|^{-1}$, we may rewrite the estimate for $\bar{\bar\phi}_1$ as 
\[|\bar{\bar\phi}_1|\leq C \mu^{\beta}(1+|y|^{n-2\beta})\]
where $\beta:=\frac32-\frac1{n-1}+\epsilon$, and $\epsilon$ is a small positive number. Plugging the estimate into the formula, we have
\[\left|\int f_1(y)\bar Z_1(y)\right|\leq Ck^{2-n-2\epsilon}\int\frac{|y_1|}{(1+|y|)^{n+6-2\beta}}\leq Ck^{2-n-2\beta}.\]
Therefore, this is a higher order term compared to the main term.

Analogously, we have 
\[\left|\int f_2(y)\bar Z_1(y)\right|\leq C\int_{B^C_{2\delta (k\mu)^{-1}}}U^{p-2}|Z_1||\bar{\bar\phi}_1|\leq Ck^{-(n+2)}.\]

The rest relies on the decomposition given in section 5. For $f_3$, the pointwise estimates and decomposition, also $\|\cdot\|_{**}$-norm estimates of $\psi$ and $\phi$, yields
\begin{equation*}
\begin{aligned}
&\left|\int f_3(y)\bar Z_1(y)\right|\leq C\mu^{\frac{n-2}2}\left|\int_{B_{\delta (k\mu)^{-1}}} U^{p-1}\psi(\bar\xi_1+\mu y)Z_1\right|\\
\leq&~C\mu^{\frac{n-2}2}\left|\int_{B_{\delta (k\mu)^{-1}}} U^{p-1}\psi^s(y)Z_1\right|+C\mu^{\frac{n-2}2}\left|\int_{B_{\delta (k\mu)^{-1}}} U^{p-1}\psi^*(y)Z_1\right|\\
=&~C\mu^{\frac{n-2}2}\left|\int_{B_{\delta (k\mu)^{-1}}} U^{p-1}\psi^*(y)Z_1\right|\\
\leq&~C\mu^\frac n2(\|\bar{\bar\phi}_1\|_{*}+\|\psi\|_*+do(1))\int U^{p-1}y_1Z_1\leq Ck^{1-n-\frac nq}
\end{aligned}
\end{equation*}
where $\epsilon$ is a positive number related to $q$ and $n$.\par
For $f_4$, since $\phi<<u_*$, we have 
\(f_4\leq C\bar\zeta_1|u_*|^{p-2}\phi^2.\) Thus the estimate is given by
\begin{equation*}
\begin{aligned}
&\left|\int f_4(y)\bar Z_1(y)\right|\leq \int CU_1^{p-2}\bar\zeta_1\left|\sum_{j=1}^k(\bar\phi_j+\hat\phi_j)+\psi\right|^2\bar Z_1\\
\leq\ &C\left\|\bar\zeta_1\bigg(\frac{y-\bar\xi_1}{\mu}\bigg)(\bar{\bar\phi}_1+\bar\psi)\right\|_*\int _{B_{\delta (k\mu)^{-1}}}CU^{p-1}\left|\bar{\bar\phi}_1+\bar\psi\right|Z_1\\
\leq\ &Ck^{-\frac nq}\int _{B_{\delta (k\mu)^{-1}}}U^{p-1}\left(|\bar{\bar{\phi}}_1^*|+\mu^{\frac{n-2}2}|\psi^*|\right) Z_1\leq Ck^{1-n-\frac nq}
\end{aligned}
\end{equation*}
Therefore, by choosing proper $q$, we have $\left|\int f_4(y)\bar Z_1(y)\right|=o(k^{1-n-\frac nq})$.
Now that we have concluded the proof for $n\geq5$, we will deal with the cases in which $n=3,4$.
For $f_3,f_4$, the proof is identical to that presented in \cite{MEDINA2021}. The differences are that different estimates, i.e., those for the lower dimensions, as well as the summation formulae, shall be used instead of the one we used in the previous proof, while the proof for symmetry decomposition is also used. We omit it here.

Finally, we only need to show that the solution achieves maximal rank in dimension 3. The calculation is simply a repeat of the calculation in \cite{MEDINA2021}, Appendix B. In fact, the proof shows that 
\[\dim T_u(\text{Orb}_Q(\text{Mob}(n)))=4n-2.\]

\begin{remark}
The final proof for the $u_{*,0}$ case only needs a tautological process with the by substituting the parameter of $u_{*,1}$ by that of $u_{*,0}$, and the process are identical. 
\end{remark}

\section*{Acknowledgment}

The research of L. Sun is partially supported by Strategic Priority Research Program of the Chinese Academy of Sciences (No.\,XDB0510201), CAS Project for Young Scientists in Basic Research Grant (No.\,YSBR-031), National Key R\&D Program of China (No.\,2022YFA1005601), NSFC China (No.\,12471115).

\bibliographystyle{plainnat}
\bibliography{ref}

\end{document}